\documentclass[a4paper]{article}
\usepackage[all]{xy}\usepackage[latin1]{inputenc}        
\usepackage[dvips]{graphics,graphicx}
\usepackage{amsfonts,amssymb,amsmath,color,mathrsfs, amstext}
\usepackage{amsbsy, amsopn, amscd, amsxtra, amsthm,authblk}
\usepackage{enumerate,algorithmicx,algorithm}
\usepackage{algpseudocode}
\usepackage{upref}
\usepackage{geometry}
\usepackage[displaymath]{lineno}
\usepackage{float}
\usepackage{yhmath}
\usepackage{booktabs}
\usepackage{subcaption}
\usepackage{multirow}
\usepackage{makecell}
\usepackage[colorlinks,
            linkcolor=blue,
            anchorcolor=blue,
            citecolor=blue
            ]{hyperref}

\numberwithin{equation}{section}

\newcommand{\bbR}{\mathbb{R}}

\newtheorem{theorem}{Theorem}[section]
\newtheorem*{theorem*}{Theorem}
\newtheorem{lemma}{Lemma}[section]
\newtheorem{definition}{Definition}[section]
\newtheorem{remark}{Remark}[section]
\newtheorem{proposition}{Proposition}[section]

\newtheorem{corollary}{Corollary}[section]

\numberwithin{equation}{section}

\begin{document}

\title{Complete monotonicity-preserving numerical methods for time fractional ODEs}

\author[1]{Lei Li\thanks{E-mail: leili2010@sjtu.edu.cn}}
\author[2]{Dongling Wang\thanks{E-mail: wdymath@nwu.edu.cn; Corresponding author.}}
\affil[1]{School of Mathematical Sciences, Institute of Natural Sciences, MOE-LSC, Shanghai Jiao Tong University, Shanghai, 200240, P. R. China.}
\affil[2]{Department of Mathematics and Center for Nonlinear Studies, Northwest University, Xi'an, Shaanxi, 710127, P. R. China.}

\date{}
\maketitle

\begin{abstract}
The time fractional ODEs are equivalent to convolutional Volterra integral equations with completely monotone kernels. We introduce the concept of complete monotonicity-preserving ($\mathcal{CM}$-preserving) numerical methods for fractional ODEs, in which the discrete convolutional kernels inherit the $\mathcal{CM}$ property as the continuous equations.  We prove that $\mathcal{CM}$-preserving schemes are at least $A(\pi/2)$ stable and can preserve the monotonicity of solutions to scalar nonlinear autonomous fractional ODEs, both of which are novel. Significantly, by improving a result of Li and Liu (Quart. Appl. Math., 76(1):189-198, 2018), we show that the $\mathcal{L}$1 scheme is $\mathcal{CM}$-preserving.
The good signs of the coefficients for such class of schemes ensure the discrete fractional comparison principles, and allow us to establish the convergence in a unified framework when applied to time fractional sub-diffusion equations and fractional ODEs.
 The main tools in the analysis are a characterization of convolution inverses for completely monotone sequences and a characterization of completely monotone sequences using Pick functions due to Liu and Pego (Trans. Amer. Math. Soc. 368(12): 8499-8518, 2016). The results for fractional ODEs are extended to $\mathcal{CM}$-preserving numerical methods for Volterra integral equations with general completely monotone kernels. Numerical examples are presented to illustrate the main theoretical results.  
\end{abstract}

\section{Introduction}\label{sec:Introd}
Fractional differential equations have received various applications in engineering and physics due to their nonlocal nature and their ability for modeling long tail memory effects \cite{brunner2017volterra, diethelm10,petravs2011fractional}. 
Compared to classical integer differential equations, time fractional differential equations, including fractional ODEs and PDEs, have two typical characteristics.  Firstly, the solutions of fractional equations usually have low regularity at the initial time \cite{brunner2017volterra, diethelm10,stynes2017error}. Secondly, the solutions of fractional equations usually have algebraic decay rate for dissipative problems which leads to the so-called long tail effect, while the solutions of classical integer equations usually have exponential decay for such problems \cite{wang2018long,vergara2015optimal,zeng2018new}. 
Because of the slow long time decay rate of the solutions of time fractional equations such that they are more advantageous than the integer order differential equations in describing many models with memory effects. 

These two features of time fractional order differential equations bring new challenges to their numerical solutions. The low regularity of the solutions at the initial time often leads to convergence order reduction in the numerical solutions.  Several technologies are developed to recover the high convergence order of numerical solutions, including adding starting weights \cite{lubich1986}, correction in initial steps \cite{yan2018analysis, jin2017correction} or non-uniform grid methods \cite{kopteva2019error,liao2019discrete,stynes2017error,li2019linearized}. For the numerical solutions that can accurately preserve the corresponding long term algebraic decay rate of the solutions of continuous equations, \cite{cuesta2007asymptotic} and \cite{wang2018long} made some first attempts for linear fractional PDEs and for nonlinear fractional ODEs respectively.

We consider the Caputo fractional ODE of order $\alpha\in(0, 1)$ for $t\mapsto u(t)\in \mathbb{R}^d$
\begin{gather}\label{eq1}
\mathcal{D}_c^{\alpha}u(t) =f(t, u(t)), \quad t>0,
\end{gather}
with initial value $u(0)=u_{0}$, where $\mathcal{D}_c^{\alpha}u(t):=\frac{1}{\Gamma(1-\alpha)}\int_{0}^{t} \frac{u'(s)}{(t-s)^{\alpha}}ds$ stands for the Caputo fractional derivatives and $f(\cdot, \cdot)$ is some given function. 
It is well known that under some suitable regularity assumptions the Caputo fractional ODE is equivalent to Volterra integral equation of the second class (see, for example, \cite[Lemma 2.3]{df02})
\begin{gather}\label{eq2}
u(t)=u_{0}+\mathcal{J}_{t}^{\alpha}f(\cdot, u(\cdot)):= u_{0}+\frac{1}{\Gamma(\alpha)}\int_{0}^{t} \frac{f(s, u(s))}{(t-s)^{1-\alpha}}ds,  \quad t>0.
\end{gather}
where 
\[
\mathcal{J}_{t}^{\alpha}g(t)=(k_{\alpha}*(\theta g)) (t)=\frac{1}{\Gamma(\alpha)}\int_{0}^{t} \frac{g(s)}{(t-s)^{1-\alpha}}ds ~\hbox{with~ kernel}~ k_{\alpha}(t)=\frac{t_+^{\alpha-1}} {\Gamma (\alpha)}
\]
 denotes the Riemann-Liouville fractional integral of order $\alpha$. Here, $\theta$ is the standard Heaviside function and $t_+=\theta(t)t$. 
In \cite{liliu2018}, a generalized definition of Caputo derivative based on convolution groups was proposed, and it has further been generalized in \cite{liliu2018compact} to weak Caputo derivatives for mappings into Banach spaces. The generalized definition, though appearing complicated, is theoretically more convenient, since it allows one to take advantage of the underlying group structure. In fact, making use of the convolutional group structure (see \cite{liliu2018} for more details), it is straightforward to convert a differential form like \eqref{eq1} into the Volterra integral like \eqref{eq2} even for $f$ to be distributions.
 
It is noted that the standard kernel function $k_{\alpha}(t)$ completely determines the basic properties of the Volterra integral equation (\ref{eq2}), so does the fractional ODE  (\ref{eq1}). Therefore, when we construct the numerical methods for equation (\ref{eq1}) or (\ref{eq2}), it's very natural and interesting to take into account some important properties of the kernel function. The standard kernel function $k_{\alpha}(t)$ represents a very important and typical class of completely monotonic ($\mathcal{CM}$) functions. Therefore, from the viewpoint of structure-preserving algorithms, it is quite natural to require the corresponding numerical methods can share this $\mathcal{CM}$ characteristic at the discrete level. This motives us to introduce the $\mathcal{CM}$-preserving numerical methods for Volterra integral equations (\ref{eq2}), in which the discrete kernel function in the corresponding numerical methods is a $\mathcal{CM}$ sequence. See the exact definition and some more explanations below in Section \ref{sec:Schemes}.

For a class of Volterra equations with completely monotonic convolution kernels, Xu in \cite{da2002uniform,da2008uniform} studied the time discretization method based on the backward Euler and convolution quadrature and established the stability and convergence in $L^{1}(0,\infty; H)\cap L^{\infty}(0,\infty; H)$ norm, where $H$ is a real Hilbert space. 
These nice works emphasize the qualitative characteristics of the solutions in the sense of average over the whole time region, which is quite different from the pointwise properties we will establish next. 

We now briefly review some basic notations for the $\mathcal{CM}$ functions and $\mathcal{CM}$ sequences 
and some related results which will be used in our later analysis, see the details in \cite{gripenberg1990volterra}. 
A function $g : (0, \infty)\to \mathbb{R}$ is called $\mathcal{CM}$ if it is of class $C^{\infty}$ and satisfies that 
\begin{gather}
(-1)^{n}g^{(n)}(t)\geq 0~ ~\hbox{for~ all}~ t>0, n=0,1,....
\end{gather}
The $\mathcal{CM}$ functions appear naturally in the models of relaxation and diffusion processes due to the fading memory principle and causality \cite{piero1997concepts}. 
In the linear viscoelasticity, a fundamental role is played by the interconversion relationships, which is modeled by a convolution quadrature with completely monotone kernels \cite{loy2014interconversion}. The $\mathcal{CM}$ functions also play a role in potential theory, probability theory and physics.
Very recently, the authors in \cite{bonaccorsi2012optimal} concerned with a class of stochastic Volterra integro-differential problem with completely monotone kernels, and use the approach to control a system whose dynamic is perturbed by the memory term.
We say a sequence $v=(v_0, v_1, \ldots)$ is $\mathcal{CM}$ if 
\begin{gather}\label{eq:CMsequencedef}
((I-E)^jv)_k\ge 0,~\text{ for any } j\ge 0, k\ge 0
\end{gather}
 where $(Ev)_j=v_{j+1}$. 
 A sequence is $\mathcal{CM}$ if and only if it is the moment sequence of a Hausdorff measure (a finite nonnegative measure on $[0,1]$) \cite{widder41}. 
 Another description we use heavily in this paper is that  a sequence is $\mathcal{CM}$ if and only if its generating function is a Pick function and analytic, nonnegative on $(-\infty, 1)$ (see Lemma \ref{lem:generating} below for more details).
 
In this paper, we first of all improve a result in \cite{li2018note} to show that the $\mathcal{L}$1 scheme (see Section \ref{sec:fourschemes} for more details) is $\mathcal{CM}$-preserving.
\begin{theorem*}[Informal version of Theorem \ref{thm:CMequivalent} and Proposition \ref{pro:L1cm}]
A sequence $a=(a_0, \cdots)$ with $a_0>0$ is $\mathcal{CM}$ if and only if its convolution inverse $\omega=a^{(-1)}$ satisfies that $\omega_0>0$, that the sequence $(-\omega_1, -\omega_2, \cdots)$
is $\mathcal{CM}$ and that $\omega_0+\sum_{j=1}^{\infty}\omega_j\ge 0$. Consequently, the $\mathcal{L}1$ scheme is  $\mathcal{CM}$-preserving.
\end{theorem*}
Of course, there are many other $\mathcal{CM}$-preserving schemes as we will discuss later. This result also tells us that the $\mathcal{CM}$-preserving schemes have nice sign properties for the coefficients: all $a_j$ for $j\geq 0$ are nonnegative and all $\omega_j$ for $j\ge 1$ are nonpositive. These allow us to establish some comparison principles and good stability properties of the schemes (see Section \ref{sec:generalproperties} for more details). In fact, by a deep  characterisation of $\mathcal{CM}$ sequences using Pick functions in \cite{lp16}, we can show a much better result: all $\mathcal{CM}$-preserving schemes are at least $A(\pi/2)$ stable.
\begin{theorem*}[Informal version of Theorem \ref{thm:stabilityregion} and Corollary \ref{cor:stability}]
Consider a $\mathcal{CM}$-preserving scheme for \eqref{eq1}. The complement of the numerical stability region is a bounded set in the right half complex plane. The stability region contains the left half plane excluding $\{0\}$, and also 
the small wedge region conducts vertex at $\{0\}$ with asymptotic angle $\pm \alpha\pi/2$. Consequently, for $\mathcal{D}_c^{\alpha}u=\lambda u$,
 the $\mathcal{CM}$-preserving schemes are unconditionally stable when $|\arg(\lambda)|\ge \pi/2$, while stable for $h$ small enough when $|\arg(\lambda)|>\frac{\pi \alpha}{2}$.
\end{theorem*}
Note that the brach cut of the $\arg(\cdot)$ function in this paper is taken to be the negative real axis and thus the range is $(-\pi, \pi]$. It is a curious question whether the numerical solutions are monotone. The monotonicity of numerical solutions is very important for proving stability of some fractional PDEs using discretized sequence to approximate. In fact, for autonomous scalar ODEs, we are able to show this.
\begin{theorem*}[Informal version of Theorem \ref{thm:generalmonotone}]
Consider applying $\mathcal{CM}$-preserving schemes to fractional ODEs $\mathcal{D}_c^{\alpha}u=f(u)$ for $f: \bbR\to \bbR$. If $f(\cdot)$ is $C^1$ and non-increasing, or  $f(\cdot)$ is $C^1$ with $M:=\sup |f'(u)|<\infty$, then for suitably chosen $h_0$, when $h\le h_0$, $\{u_n\}$ is monotone.
\end{theorem*}

By the the good signs of the the sequence $a$ and $\omega$, we are able to establish the convergence of the numerical solutions to fractional ODEs for $\mathcal{CM}$-preserving schemes in a unified framework.
\begin{theorem*}[Informal version of Theorem \ref{thm:convergece}]
Consider applying $\mathcal{CM}$-preserving schemes to frational ODEs
$\mathcal{D}_c^{\alpha}u=f(t, u)$, where $u: [0, T]\to \mathbb{R}^d$. If $f(t, \cdot)$ satisfies $(x-y)\cdot (f(t,x)-f(t, y))\le 0$ or is Lipschitz continuous,
then,
\begin{gather}
\lim_{h\to 0}\sup_{n: nh\le T}\|u(t_n)-u_n\|=0.
\end{gather}
\end{theorem*}
We also apply similar techniques to Volterra convolutional integral equations and obtain similar results, which we do not list here.

The rest of this paper is organized as follows. In Section \ref{sec:Schemes}, 
we first provide the motivations for $\mathcal{CM}$-preserving numerical schemes for fractional ODEs and then give the exact definition. 
In Subsection \ref{sec:generalproperties}, we show that the condition for the inverse of a $\mathcal{CM}$ sequence in \cite{li2018note} is in fact both necessary and sufficient. Some favorable properties such as discrete fractional comparison principles for $\mathcal{CM}$-preserving numerical schemes are derived. Four concrete numerical schemes, including the Gr\"{u}nwald-Letnikov formula, numerical method based on piecewise interpolation, convolutional quadrature based on  $\theta$-method and the $\mathcal{L}$1 scheme are shown to be $\mathcal{CM}$-preserving for fractional ODEs in Subsection \ref{sec:fourschemes}.
In Section \ref{sec:stability}, we study the stability region for general $\mathcal{CM}$-preserving schemes and prove 
they are $A(\pi/2)$-stable. The new results allow us to apply $\mathcal{CM}$-preserving schemes to linear systems where the eigenvalues may have non-zero imaginary parts but still maintain numerical stability.
The monotonicity of numerical solutions obtained by $\mathcal{CM}$-preserving numerical methods for scalar nonlinear autonomous fractional ODEs is proved in Section \ref{sec:scalar}, which is fully consistent with the continuous equations.  In Section \ref{sec:subdiffusion}, we first derive the local truncation error and convergence of $\mathcal{CM}$-preserving schemes for fractional ODEs. Then we apply $\mathcal{CM}$-preserving schemes to time fractional sub-diffusion equations, in which we are able to establish the convergence of the numerical methods in time direction in a unified framework due to the nice sign properties of the $\mathcal{CM}$-preserving schemes. This new class of numerical methods for fractional ODEs are directly extended to convolutional Volterra integral equations involving general $\mathcal{CM}$ kernel functions in Section \ref{Volterra}. Several numerical examples and concluding remarks are included in Section \ref{numerical}.

\section{$\mathcal{CM}$-preserving numerical schemes for fractional ODEs}
\label{sec:Schemes}

Let us consider the fractional ODE \eqref{eq1} of order $\alpha\in (0, 1)$, 
subject to $u(0)=u_{0}>0$.
Consider the implicit scheme approximating $u(t_{n})$ by $u_{n}$ ($n\geq 1$) at the uniform grids $t_{n}=nh$ with step size $h>0$ of the following form:
\begin{gather}\label{eq22}
(\mathcal{D}_h^{\alpha}u)_{n} :=h^{-\alpha}\sum_{j=0}^n \omega_{j}(u_{n-j}-u_0)=f(t_{n}, u_n):=f_n,~~n\ge 1.
\end{gather}
If we would like to include $n=0$, \eqref{eq22} is written as
\begin{gather}\label{eq23}
h^{-\alpha}\sum_{j=0}^n \omega_{j}(u_{n-j}-u_0)= f_n-f_0\delta_{n,0},~~n\ge 0,
\end{gather}
where $\delta_{i,j}=1$ if $i=j$ and $\delta_{i,j}=0$ if $i\neq j$ is the usual Kronecker function so that $\delta_{n,0}$ is the $n$th entry of the convolutional identity $\delta_{d}:=(1, 0, 0,...).$

\begin{remark}
Note that we understand $\mathcal{D}_h^{\alpha}u$ in \eqref{eq22} as a sequence, and thus $(\mathcal{D}_h^{\alpha}u)_{n}$ means the $n$th term in the sequence. Later, we sometimes use sloppy notations like $\mathcal{D}_h^{\alpha}u_{n}$ or $\mathcal{D}_h^{\alpha}f(u_n)$ to mean the $n$th term of the sequence obtained by applying $\mathcal{D}_h^{\alpha}$ on the sequence $(u_n)$ or $(f(u_n))$. (It does not mean the operator acting on the constant $u_n$ or $f(u_n)$.)
\end{remark}

The convolution inverse of $\omega$ is defined by $a=\omega^{(-1)}$ such that $\omega*\omega^{(-1)}=\omega^{(-1)}*\omega=\delta_{d}$.
Let us introduce generating function of a sequence $v=(v_0, v_1, \ldots)$, defined by 
\begin{gather}\label{eqx3}
F_v(z)=\sum_{n=0}^{\infty} v_n z^n, \quad z\in\mathbb{C}.
\end{gather}
The generating function should be  understood in the sense of analytic continuation.  We choose the continuation that has the largest possible domain in the upper half plane and symmetric about the real axis \cite{fornberg2019complex}. 
For example, the generating function of the sequence $(1, 1, ...)$ is given by $F_1(z):=\frac{1}{1-z}$, which is defined in the entire plane except $z=1$.

It is straightforward to verify that $F_{u*v}(z)=F_u(z)F_v(z)$. Hence, the generating functions of $a$ and $\omega$ are related by
$F_a(z)=\frac{1}{F_{\omega}(z)}$.  By the convolution inverse, the above numerical scheme (\ref{eq23}) can be written as
\begin{gather}\label{eq24}
u_n-u_0=h^{\alpha}[a*(f- f_0 \delta_{d})]_n=h^{\alpha}[a*f-f_0 a]_{n}=h^{\alpha}\sum_{j=0}^{n-1} a_j f_{n-j}, ~~n\ge 1,
\end{gather}
Hence, $\{a\}$ given in the numerical scheme can be regarded as some integral discretization of the fractional integral.

Following \cite{lubich1986}, we define
\begin{definition}\label{def:consistency}
We say discretization \eqref{eq23} or  \eqref{eq24} is consistent if 
$h^{\alpha}F_a(e^{-h})=1+o(1),~h\to 0^+$.
\end{definition}

Since the kernel $k_{\alpha}(t)=\frac{t_+^{\alpha-1}}{\Gamma(\alpha)}$ involved in the Riemann-Liouville fractional integral is a typical $\mathcal{CM}$ function, 
from the structure-preserving algorithm point of view,  it is natural to desire the corresponding numerical methods can inherit this key property at the discrete level. We are then motivated to define the following:
\begin{definition}
We say a consistent (in the sense of Definition \ref{def:consistency}) numerical method given in \eqref{eq22}
for the time fractional ODEs is $\mathcal{CM}$-preserving if the sequence $a=\omega^{(-1)}$ is a $\mathcal{CM}$ sequence.
\end{definition}

\subsection{General properties of $\mathcal{CM}$-preserving schemes}
\label{sec:generalproperties}

The $\mathcal{CM}$-preserving numerical schemes have many favorable properties, and we now investigate these properties. We first of all introduce the concept of Pick functions. A function $f:\mathbb{C}_+\to\mathbb{C}$ (where $\mathbb{C}_+$ denotes the upper half plane, 
not including the real line) is Pick if it is analytic such that $\mathrm{Im}(z)>0\Rightarrow \mathrm{Im}(f(z))\ge 0$. Throughout this paper, $\mathrm{Im}(z)$ and $\mathrm{Re}(z)$ denote the imaginary and real parts of $z$, respectively.
We have the following observation.
\begin{lemma}\label{lem:Pick}
If $F(z)$ is a Pick function and $\mathrm{Im}(F(z))$ achieves zero at some point in $\mathbb{C}_{+}$, then $F(z)$
is a constant.
\end{lemma}

 Let $v=\mathrm{Im} F(z)$. Then $v$ is a harmonic function and $v\ge 0$. If $v$ achieves the minimum $0$ inside the domain, then it must be a constant
 by the maximal principle.
 Then, by Cauchy-Riemann equation, $\mathrm{Re}(F(z))$ is also constant and the result follows.

Now, we can state some properties of sequences in terms of the generating functions, for which we omit the proofs.
\begin{lemma}\label{lem:generating}
\begin{enumerate}[(1)]
\item (\cite[Corollary VI.1]{fs09}) Assume $F_v(z)$ is analytic on $\Delta :=\{z: |z|<R, z\neq 1, |\mathrm{arg}(z-1)|>\theta\}$, for some $R>1, \theta\in (0, \frac{\pi}{2})$. 
If $F_v(z)\sim (1-z)^{-\beta}$ as $z\to 1, z\in \Delta$ for $\beta\neq 0, -1, -2, -3, \cdots$, then $v_n \sim \frac{1}{\Gamma(\beta)}n^{\beta-1},~n\to \infty$.

\item $\lim_{n\to\infty}v_n=\lim_{z\to 1^-}(1-z)F_v(z)$.

\item (\cite{lp16}) A sequence $v$ is $\mathcal{CM}$ if and only if the generating function $F_v(z)=\sum_{j=0}^{\infty}v_jz^j$ is a Pick function that is analytic and nonnegative on $(-\infty, 1)$. 
\end{enumerate}
\end{lemma}

In \cite{li2018note}, Li and Liu have proved that for a given $\mathcal{CM}$ sequence $a$ with $a_0>0$, the inverse sequence $\omega=a^{-1}$ has very nice sign consistency condition:
\begin{equation} \label{eq:signconsis}
\begin{split}
\hbox{(i):}~ \omega_{0}>0,~ \omega_{j}\leq 0~\hbox{for}~ j\geq1;\quad  \hbox{(ii):}~ \omega_{0}+\sum_{j=1}^{\infty}\omega_{j}\ge 0.
\end{split}
\end{equation}
When $\|a\|_{\ell^1}=\infty$, the last inequality becomes equality, which is the case for schemes of time fractional ODEs.

According this result, one is curious about the converse of the result: given
$\omega=(\omega_0, \omega_1, \cdots)$ with $\omega_0>0$, the sequence $(-\omega_1,-\omega_2,\cdots )$ to be $\mathcal{CM}$
and that  $\omega_0+\sum_{j=1}^{\infty}\omega_j\ge 0$,
can we have the convolutional inverse $a=\omega^{-1}$  to be also a $\mathcal{CM}$ sequence?
This is particularly interesting regarding $\mathcal{L}$1 scheme (see section \ref{sec:fourschemes} for more details). In $\mathcal{L}$1 scheme, we get a discrete convolutional scheme $\omega$, which is an approximation for the Caputo fractional derivative. By taking the inverse of $a=\omega^{(-1)}$, we then get a corresponding discrete convolutional scheme which is an approximation for the fractional integral, and what we need to do is verify that $a$ is a $\mathcal{CM}$ sequence.

In this subsection, we would like to establish our first main result, i.e., the converse of the Theorem 3.2 in \cite{li2018note} is also correct, that is, to establish a sufficient and necessary condition for the convolutional inverse of a  $\mathcal{CM}$ sequence. As an application of this results,  we will show in section \ref{sec:fourschemes} that the well known $\mathcal{L}$1 scheme is $\mathcal{CM}$-preserving.

\begin{theorem}\label{thm:CMequivalent}
The sequence $a=(a_0, \cdots)$ with $a_0>0$ is $\mathcal{CM}$ if and only if its convolution inverse $\omega=a^{-1}$ satisfies that $\omega_0>0$, that the sequence $(-\omega_1, -\omega_2, \cdots)$
is $\mathcal{CM}$ and that $\omega_0+\sum_{j=1}^{\infty}\omega_j\ge 0$. Moreover, $\omega_0+\sum_{j=1}^{\infty}\omega_j=\|a\|_{\ell^1}^{-1}$.
\end{theorem}

\begin{proof}
The ``$\Rightarrow$'' direction has been proved in Theorem 3.2 in \cite{li2018note}. We now prove the reverse direction. 

Define the generating function for sequence $(-\omega_1, -\omega_2, \cdots)$ by
\begin{gather}\label{eq:Gfun}
G(z)=\sum_{j=0}^{\infty}(-\omega_{j+1})z^j=\sum_{j=1}^{\infty}(-\omega_j)z^{j-1}.
\end{gather}
Hence, one has $F_{\omega}(z)=\omega_0-z G(z)$.
By Lemma \ref{lem:generating}, $G(z)$ is a Pick function that is nonnegative and analytic on $(-\infty, 1)$. 
We now investigate the generating function of $a$:
\[
F_a(z)=F_{\omega}^{-1}(z)=\frac{1}{\omega_0-z G(z)}.
\]

To do this, for $\epsilon>0$ we consider an auxiliary function given by
\begin{gather}\label{eq:PHfun}
H_{\epsilon}(z)=\frac{1}{\epsilon}+\frac{z}{\epsilon+\omega_0-z(\epsilon+G(z))}=\frac{\epsilon+\omega_0-zG(z)}{\epsilon(\epsilon+\omega_0-z(\epsilon+G(z)))}.
\end{gather}
Since both $G(z)$ and $\epsilon+G(z)$ are nonnegative on $(-\infty, 1)$, one finds that
\[
\epsilon+\omega_0-z(\epsilon+G(z))>0, ~\epsilon+\omega_0-zG(z)>0
\]
for $z\le 0$. For $z\in (0, 1)$, it is then clear
\[
\epsilon+\omega_0-z(\epsilon+G(z))> \epsilon+\omega_0-(\epsilon+G(1))
=\omega_0-G(1)\ge 0.
\]
Similarly
\[
\epsilon+\omega_0-zG(z)\ge \epsilon+\omega_0-G(z)>0.
\]
Hence, $H_{\epsilon}(z)$ is nonnegative on $(-\infty, 1)$. 
The argument here also justifies that $A_{\epsilon}(z):=\epsilon+\omega_0-z(\epsilon+G(z))$ is never zero on $(-\infty, 1)$. Moreover, for $z\in \mathbb{C}_{+}$, the phase of $\epsilon+G(z)$ is in $(0, \pi)$, and thus
$z(\epsilon+G(z))$ cannot be a real positive number. Hence, $A_{\epsilon}(z)$ is never zero in the upper half plane so that $H_{\epsilon}(z)$ is analytic on $\mathbb{C}_{+}\cup(-\infty, 1)$. Moreover,
\[
\frac{z}{\epsilon+\omega_0-z(\epsilon+G(z))}=\frac{z(\epsilon+\omega_0)-|z|^2 \left(\epsilon+\overline{G(z)} \right)} {|\epsilon+\omega_0-z(\epsilon+G(z))|^2}.
\]
It follows from $\mathrm{Im}(z)>0 \Rightarrow \mathrm{Im}(G(z))\ge 0$ that $\mathrm{Im} \left( \overline{G(z)} \right)\le 0$ for $\mathrm{Im}(z)>0$. We find that 
$H_{\epsilon}(z)$ is a Pick function.
Hence, the sequence
\begin{gather}\label{eq:Psequence}
\left( \frac{1}{\epsilon}, a_0(\epsilon), a_1(\epsilon), \cdots \right)
\end{gather}
corresponding to the generating function $H_{\epsilon}(z)$ is $\mathcal{CM}$.

By the definition (equation \eqref{eq:CMsequencedef}), $\left(a_0(\epsilon), a_1(\epsilon), \cdots\right)$ is also $\mathcal{CM}$. This sequence corresponds to the generating function
\begin{gather}\label{eq:PFfun}
F_{a(\epsilon)}(z)=\frac{1}{\omega_0+\epsilon-z(\epsilon+G(z))},
\end{gather}
which must be Pick and nonnegative on $(-\infty, 1)$ by Lemma \ref{lem:generating}(3).
We first note that $F_a(z)=\frac{1}{\omega_0-zG(z)}$ is analytic in $\mathbb{C}_{+}\cup(-\infty, 1)$ by similar argument. Then, taking $\epsilon\to 0^{+}$, as the pointwise limit of $F_{a(\epsilon)}(z)$, $F_a(z)$ must also be Pick and nonnegative on $(-\infty, 1)$. Hence,
$(a_0, a_1, \cdots)$ is $\mathcal{CM}$ by Lemma \ref{lem:generating}(3). 

Regarding the equality $\omega_0+\sum_{j=1}^{\infty}\omega_j=\|a\|_{\ell^1}^{-1}$, we just note $F_a(z)=F_{\omega}^{-1}(z)$, take $z\to 1^-$ and apply the monototone convergence theorem due to signs of $a_j$'s and $\omega_j$'s.
\end{proof}

With results in Lemma \ref{lem:generating} and Theorem \ref{thm:CMequivalent}, we are able to establish a series of basic properties of $\mathcal{CM}$-preserving schemes. The first result is as follows.
\begin{proposition}\label{pro:propertyomega}
If the discretization is $\mathcal{CM}$-preserving with $a_0>0$, then
\begin{gather}
\begin{split}
a_j\sim \frac{1}{\Gamma(\alpha)}j^{\alpha-1},~j\to\infty,~~
 h^{\alpha}\sum_{j=1}^na_j\le C(nh)^{\alpha},~\forall n.
\end{split}
\end{gather}
Moreover, the convolutional inverse $\omega$ satisfies: $\omega_0>0$  and $\omega_j \le 0$ for all $j=1,2,\cdots$ and
$\omega_0+\sum_{j=1}^{\infty}\omega_j=0.$
The generating function is given by $F_{\omega}(z)=(1+o(1))(1-z)^{\alpha}$,~$z\to 1$ so that
$\omega_j\sim \frac{1}{\Gamma(-\alpha)}j^{-1-\alpha},~j\to\infty.$
\end{proposition}
Definition \ref{def:consistency} directly means $F_a(z)=(1+o(1))(1-z)^{-\alpha}$ as $z\to 1^-$. The generating function of the sequence $\{A_n:=\sum_{j=0}^n a_j\}_{n=0}^{\infty}$
is $(1-z)^{1-\alpha}(1+o(1))$.
Moreover, since $F_a(z)$ is a Pick function with $a_0>0$, then $F_a(z)$ is analytic in $\mathbb{C}_+$ without zeros in the upper half plane. The claims then follow directly from Lemma \ref{lem:generating} and Theorem \ref{thm:CMequivalent}. We omit the details.

This good sign invariant property in the coefficients of $\{\omega_{j}\}$ plays a key role in energy methods for numerical analysis \cite{wang2018long, wang2019dis,li2019discretization}. One obvious observation is
 \begin{proposition}\label{pro:energy}
Assume the scheme for the discrete Caputo operator $\mathcal{D}^{\alpha}_{h}$ in (\ref{eq22}) is $\mathcal{CM}$-preserving. Consider that $E(\cdot): \mathbb{R}^d\to \mathbb{R}$ is a convex function. Then, we have
\begin{gather}
\mathcal{D}_h^{\alpha} E(u_n)\le \nabla E(u_n)\cdot \mathcal{D}_h^{\alpha}u_n.
\end{gather}
 \end{proposition}
For the proof, one may make use of the fact that $\omega_0+\sum_{j=1}^{\infty}\omega_j=0$ (due to $\|a\|_{\ell^1}=\infty$) to define $c_j=-\omega_j\ge 0$
and $\sigma_n:=\sum_{j=n}^{\infty}c_j\ge 0$, so that
\[
(\mathcal{D}_h^{\alpha}u)_n=h^{-\alpha}\left(\sum_{j=1}^{n-1}c_j(u_n-u_j)+\sigma_n(u_n-u_0)\right).
\]
 The claim then follows from the convexity: $\nabla E(u_n)\cdot(u_n-u_j)
 \ge E(u_n)-E(u_j)$. We skip the details.
 
The sign properties also guarantee the discrete fractional comparison principles as follows (see \cite{li2019discretization} for relevant discussions).  

\begin{proposition}\label{pro:comppf}
Let $\mathcal{D}^{\alpha}_{h}$ be the discrete Caputo operator defined in (\ref{eq22}) and the corresponding numerical schemes are $\mathcal{CM}$-preserving.
Assume three sequences $u,v,w$ satisfy $u_0\le v_0\le w_0$.
\begin{enumerate}[(1)]
\item Suppose $f(s, \cdot)$ is non-increasing and the following discrete implicit relations hold
\[
\mathcal{D}^{\alpha}_{h}u_n \le f(t_n, u_n),~~\mathcal{D}^{\alpha}_{h}v_n = f(t_n, v_n),~~
\mathcal{D}^{\alpha}_{h}w_n \ge f(t_n, w_n).
\]
Then,  $u_n\le v_n \le w_n$.

\item Assume $f$ is Lipschitz continuous in the second variable with Lipschitz constant $L$. If
\[
\mathcal{D}^{\alpha}_{h}u_n\le f(t_n, u_n),~~\mathcal{D}^{\alpha}_{h}v_n = f(t_n, v_n),~~
\mathcal{D}^{\alpha}_{h}w_n \ge f(t_n, w_n),
\]
then for step size $h$ with $h^{\alpha}La_0<1$, $u_n\le v_n \le w_n$.

\item Assume $f(t, \cdot)$ is nondecreasing and Lipschitz continuous in the second variable with Lipschitz constant $L$. If for $h$ with $h^{\alpha}La_0<1$,
\begin{gather*}
\begin{split}
&u_n\leq u_0+h^\alpha\sum_{j=0}^{n-1}a_{j} f(t_{n-j}, u_{n-j}), v_n= v_0+h^\alpha\sum_{j=0}^{n-1}a_{j}f(t_{n-j}, v_{n-j}),\\
&w_n\geq w_0+h^\alpha\sum_{j=0}^{n-1}a_{j}f(t_{n-j}, w_{n-j}),
\end{split}
\end{gather*}
then $u_n\le v_n \le w_n.$

\end{enumerate}
\end{proposition}

The proof is similar to the ones in \cite{li2019discretization}, and we give some brief proofs in Appendix \ref{app:compproof}.

\subsection{Four $\mathcal{CM}$-preserving numerical schemes}
\label{sec:fourschemes}

In this subsection, we identify several concrete $\mathcal{CM}$-preserving numerical schemes. We need to verify that the sequence $a=\{a_j\}$ is a $\mathcal{CM}$ sequence. One can either check this directly using definition (equation \eqref{eq:CMsequencedef}), use Theorem \ref{thm:CMequivalent} or check if the generating functions $F_{a}(z)$ is a Pick function or not and the non-negativity on $(-\infty, 1)$ according to Lemma \ref{lem:generating}. 

\subsubsection{The Gr\"{u}nwald-Letnikov (GL) scheme} 

Consider the Gr\"{u}nwald-Letnikov (GL) scheme for approximating of Riemann-Liouville fractional derivative \cite{diethelm10},
 whose generating function is 
$F_{\omega}(z)=(1-z)^{\alpha},$
where we recall that the branch cut for the mapping $w\mapsto w^{\alpha}$ is taken to be the negative real axis.
Hence,
\begin{gather}\label{eq25}
F_{a}(z)=(1-z)^{-\alpha}.
\end{gather}
It is easy to verify that $F_a(z)$ is a pick function and analytic, positive
on $(-\infty, 1)$.
Hence, $a$ is a $\mathcal{CM}$ sequence and the scheme \eqref{eq22} with $\{\omega_j\}$ given by the GL scheme is $\mathcal{CM}$-preserving.

\subsubsection{The $\mathcal{L}$1 scheme}

The $\mathcal{L}$1 scheme, which was independently developed and analyzed in  
\cite{sun2006fully} and \cite{lin2007finite}, can be seen the fractional generalization of the backward Euler scheme for ODEs. 
On the uniform grid $t_n=nh$ for $n=0,1,...$, the $\mathcal{L}$1 scheme for $n\geq 1$ is given by
 \begin{equation} \label{eq:L1scheme}
\begin{split}
\mathcal{D}_c^{\alpha}u(t_n)&=\frac{1}{\Gamma(1-\alpha)} \sum_{j=0}^{n-1}\int_{t_j}^{t_{j+1}} \frac{u'(s)}{(t_n-s)^{\alpha}}ds\\
&\approx \frac{1}{\Gamma(1-\alpha)} \sum_{j=0}^{n-1}\frac{u(t_{j+1})- u(t_{j})} {h}  \int_{t_j}^{t_{j+1}} \frac{1}{(t_n-s)^{\alpha}}ds\\
&= \sum_{j=0}^{n-1}b_j \frac{u(t_{n-j})- u(t_{n-j-1})} {h^{\alpha}} \\
&= \frac{1}{h^{\alpha}} \left(b_{0}u_{n}-b_{n-1}u_{0}+\sum_{j=1}^{n-1}(b_{j}-b_{j-1})u_{n-j}  \right),
\end{split}
\end{equation}
where the coefficients $b_{j}= ((j+1)^{1-\alpha}-j^{1-\alpha} )/\Gamma(2-\alpha)$, $j=0,1,2,...,n-1$. It can be written in the discrete convolution form
\[
\mathcal{D}_h^{\alpha}(u_n):=\frac{1}{h^{\alpha}} \left(\sum_{j=0}^{n-1} \omega_{j}u_{n-j}-\sigma_n u_0 \right)
= \frac{1}{h^{\alpha}} \sum_{j=0}^{n} \omega_{j}(u_{n-j}-u_0),
\] where
\begin{equation} \label{eq:L1coeff}
\begin{split}
\omega_{0}&=\frac{1}{\Gamma(2-\alpha)},\quad \sigma_{n}=b_{n-1}= \frac{1}{\Gamma(2-\alpha)} \left( n^{1-\alpha}-(n-1)^{1-\alpha}\right),\\
\omega_{j}&=\frac{1}{\Gamma(2-\alpha)} \left( (j+1)^{1-\alpha}-2j^{1-\alpha}+(j-1)^{1-\alpha}\right),\quad j\geq 1.
\end{split}
\end{equation}
One can check the coefficients $\{\omega_{j}\}$ satisfy the sign consistency condition given in \eqref{eq:signconsis} (with the last inequality being equality).
Moreover, $\sigma_n=-\sum_{j=n}^{\infty}\omega_j$.

The $\mathcal{L}$1 scheme is among the most popular and successful numerical approximations for Caputo derivatives, and is very easy to implement with acceptable precision. In \cite{jin2015analysis}, Jin et.al. strictly analyzed the convergence for both smooth and non-smooth initial data and established the optimal first order convergence rate for non-smooth data. In \cite{yan2018analysis}, Yan et.al. further provided a correction technique, in which the convergence rate for non-smooth data can be improved to $(2-\alpha)$-th order.
From \eqref{eq:L1scheme} we can see that if we consider the partition in a non-uniform grid with $h_{j}=t_{j+1}-t_j$, we can get a similar numerical scheme. This provides a good basis for various numerical approximation for Caputo derivatives on non-uniform grids, see \cite{kopteva2019error,liao2019discrete,stynes2017error,li2019linearized}.

As an application of Theorem \ref{thm:CMequivalent}, we show that 
the $\mathcal{L}$1 scheme with uniform mesh size is a $\mathcal{CM}$-preserving scheme.
\begin{proposition}\label{pro:L1cm}
For the sequence $\omega=\{\omega_j\}$ defined in \eqref{eq:L1coeff}, the convolutional inverse $a=\omega^{(-1)}$ is a $\mathcal{CM}$ sequence. Hence, the $\mathcal{L}$1 scheme is $\mathcal{CM}$-preserving.
\end{proposition}

\begin{proof}
As pointed out in \eqref{eq:signconsis}, one can directly check that $\omega_{0}>0$, 
$\omega_{j}<0$ for $j\geq 1$ and $\omega_0+\sum_{j=1}^{\infty}\omega_{j}=0$. 
We now verify that the sequence $(-\omega_1, -\omega_2, \cdots)$
 given in \eqref{eq:L1coeff} is $\mathcal{CM}$.
 In fact, from \eqref{eq:L1scheme} we know that the sequence $b=(b_0, b_1, b_2,\cdots)$ is the integral for the $\mathcal{CM}$ function $\frac{t^{-\alpha}}{\Gamma(1-\alpha)}$ on uniform mesh. That is 
 \[
 b_j=\int_{t_j}^{t_{j+1}}\frac{t^{-\alpha}}{\Gamma(1-\alpha)}\,dt,
 \] 
 so it is a $\mathcal{CM}$ sequence. 
 Then, $ \omega_{j}=b_j-b_{j-1},~j=1,2,\cdots.$
 By the definition of $\mathcal{CM}$ sequence (equation \eqref{eq:CMsequencedef}), we find $(-\omega_1, -\omega_2, \cdots)=(I-E)b$ is also a $\mathcal{CM}$ sequence. 
 Hence, by Theorem \ref{thm:CMequivalent}, the convolution inverse of $\omega$ is a $\mathcal{CM}$ sequence.
 
Lastly, it is well known that $\mathcal{L}$1 scheme is consistent and thus
\[
F_a(z)\sim (1-z)^{-\alpha},~z\to 1.
\]
In fact, this can also be proved by the aysmptotic behavior of $\omega_j$. We omit the details. This means that $\mathcal{L}$1 scheme is $\mathcal{CM}$-preserving.
 \end{proof}

\subsubsection{A scheme based on piecewise interpolation}

 Another scheme is the one in \cite{li2019discretization}. Consider the discretization of the Volterra integral form (\ref{eq2}) by approximating $f$ with piecewise constant functions, where the sequence $a$ is obtained from discretizing the
integral directly. More precisely, due to homogeneity, 
\[
a_n=h^{-\alpha}\int_{t_n}^{t_{n+1}}k_{\alpha}(s)\,ds
=\int_n^{n+1}k_{\alpha}(s)\,ds.
\]
And it can be explicitly obtained 
\[
a=(a_{0}, a_{1}, ..., a_{n},....)=\frac{1}{\Gamma(1-\alpha)} \left(1, 2^{\alpha}-1, ..., (n+1)^{\alpha}-n^\alpha,...\right).
\]
Since $t^{\alpha-1}$ is completely monotone, the sequence is as well.
Hence, the scheme \eqref{eq22} with $\{\omega_j\}=a^{(-1)}$ is $\mathcal{CM}$-preserving for \eqref{eq1}.

\subsubsection{A class of convolutional quadrature schemes}

 Consider the convolutional quadrature (CQ) proposed by Lubich \cite{lubich1986,lubich1988CQ}. 
The linear multistep methods for ODE $u'(t)=f(t, u(t))$ reads that 
\[
\sum_{j=0}^{k}\alpha_{j}u_{n+j-k}=h\sum_{j=0}^{k}\beta_{j}f_{n+j-k}.
\]
Let $\rho(z)=\sum_{j=0}^{k}\alpha_{j}z^j, \sigma(z)=\sum_{j=0}^{k}\beta_{j}z^j$ denote the generating polynomials.
The corresponding reflected polynomials \cite{lubich1983stability}
\begin{gather}\label{eq26}
\begin{split}
\breve{\rho}(z)=z^k \rho(z^{-1})=\alpha_{0}z^k+\cdots+\alpha_{k-1}z+\alpha_{k},\\
\breve{\sigma}(z)=z^k \sigma(z^{-1})=\beta_{0}z^k+\cdots+\beta_{k-1}z+\beta_{k}.
\end{split}
\end{gather}
The generating function in CQ approximating the Riemann-Liouville fractional integral \cite{lubich1986,lubich1988CQ} can be written 
\[
F_a(z)=K (\delta(z))=(\delta(z))^{-\alpha},
\]
where $K$ is the Laplace transform of the standard kernel $k_{\alpha}(t)$ and $\delta(z)=\breve{\rho}(z)/\breve{\sigma}(z)$.
Note that the GL scheme can be seen the fractional generation of back Euler method. In this scheme, we have 
$\rho(z)=z-1$ and $\sigma(z)=z$, and that $\delta(z)=\breve{\rho}(z)/\breve{\sigma}(z)=1-z$, 
which yields that $F_a(z)=(\delta(z))^{-\alpha}=(1-z)^{-\alpha}$. This is completely consistent with the formula in \eqref{eq25}.
The $\theta$-method with parameter $\theta (\theta\geq 1)$ for ODEs $u'(t)=f(t, u(t))$ reads 
$u_{n+1}=u_{n}+h((1-\theta)f_{n}+\theta f_{n+1})$. The corresponding characteristic polynomials $\rho(z)=z-1$ and $\sigma(z)=\theta z+(1-\theta)$.
For any $\theta\geq 1$, this method satisfies the consistent condition: $\rho(1)=0$ and $\rho'(1)=\sigma(1)=1$,
and $(-\infty, 0]\in \mathcal{S}_{\theta}$, where $\mathcal{S}_{\theta}$ denotes the stability region of the scheme. 
The generating function
\[
\delta(z)=\frac{1-z}{\theta +(1-\theta) z}.
\]

It is not hard to verify that for such CQ schemes, the generating function $F_a(z)$ is Pick. To do that, we write
\[
F_a(z)=\left( \frac{\theta +(1-\theta) z}{1-z} \right)^{\alpha}
:=(G(z))^{\alpha}.
\]
We claim the function $G$ is Pick. In fact,
\[
G(z)=\frac{\theta +(1-\theta) z}{1-z} =\frac{(\theta+(1-\theta)z)(1- \bar{z})}{|1- z|^2}=
\frac{\theta-\theta\bar{z}+(1-\theta)z-(1-\theta)|z|^2} {|1-z|^2},
\]
which implies that $\mathrm{Im}(G)=\mathrm{Im}( \frac{z}{|1- z|^2})$, and the result follows.
On the other hand, 
\[
\lim_{z\to -\infty}G(z)=\theta -1,
\]
which is non-negative for $\theta \geq1$. With this, when $z\in (-\infty, 1)$, $G(z)=\frac{1+(\theta-1)(1-z)}{1-z}>0$.
Hence, if $\theta\ge 1$, $G(z)$ is a Pick function that is analytic and positive on $(-\infty, 1)$ and consequently, $F_a(z)$ is also Pick and nonnegative on $(-\infty, 1)$.

As a byproduct, we know from Lemma \ref{lem:generating} that when $0\leq\theta<1$, the corresponding CQ generated by $\theta$ method is not $\mathcal{CM}$-preserving. In particular, the fractional trapezoidal method, where $\theta=1/2$, is not $\mathcal{CM}$-preserving.

\subsubsection{A comment on computation of the weights}

To close this section, we now give some comments to the computation on the weights in the expansion of $F_{\omega}(z)=\sum_{n=0}^{\infty} \omega_{n}z^{n}$. 
In general, it is not easy to evaluate the weight $\omega_{n}$ in the fractional formal power series of some polynomials. But in our case, the following Miller formula is an efficient tool. 
\begin{lemma}\label{lem:FPS}
(\cite{galeone2008fractional}) 
Let $\phi(\xi)=1+\sum_{n=1}^{\infty}c_{n}\xi^n$ be a formal power series. Then for any $\alpha\in\mathbb{C}$, 
$(\phi(\xi))^{\alpha}=\sum_{n=0}^{\infty}v_{n}^{(\alpha)}\xi^n$, where the coefficients $v_{n}^{(\alpha)}$ can be recursively evaluated as
\[
v_{0}^{(\alpha)}=1,~~ v_{n}^{(\alpha)}=\sum_{j=1}^{n}\left( \frac{(\alpha+1)j}{n}-1\right)c_{j} v_{n-j}^{(\alpha)}.
\]
\end{lemma}

Applying this lemma to the formal power series $(1\pm \xi)^\alpha=\sum_{n=0}^{\infty} \omega_{n}\xi^n$ leads to that 
\[
\omega_0=1,~~ \omega_{n}=\pm \left( \frac{(\alpha+1)}{n}-1\right) \omega_{n-1}, ~n\geq 1.
\]

With this formula and the property for the generating functions $F_{v^{(-1)}}(z)=(F_{v}(z))^{-1}$ given in Lemma \ref{lem:generating},  
We can easily calculate the weight coefficients for the schemes given in this section.

\section{Stability regions for $\mathcal{CM}$-preserving schemes}
\label{sec:stability}

It is a fundamental problem to study the stability and stability regions of numerical schemes. For the convolution quadrature approximating fractional integral based on linear multistep methods developed by Lubich \cite{lubich1985fractional,lubich1988CQ}, the stability regions were fully identified due to the inherent advantages of this kind of algorithm. The $\mathcal{L}$1 scheme can be seen a fractional generalization of backward Euler method of ODEs,  
which has been studied in various ways due to its ease of implementation, good numerical stability and acceptable computational accuracy \cite{jin2015analysis,yan2018analysis,liao2019discrete,stynes2017error,kopteva2019error}. 
The stability analysis for $\mathcal{L}$1 scheme is slightly more difficult. 
The generating functions of $\omega$ for $\mathcal{L}$1 scheme is given by 
\begin{gather}\label{eq:generatingfunL1}
F_\omega(z)=\sum_{n=0}^{\infty} \omega_n z^n=\left(\frac{1}{z}-2+z\right) \mathrm{Li}_{\alpha-1}(z), 
\end{gather}
where $\mathrm{Li}_{p}(z)$ stands for the polylogarithm function defined by $\mathrm{Li}_{p}(z)=\sum_{k=1}^{\infty}\frac{z^k}{k^p}$.
The $\mathrm{Li}_{p}(z)$ function is well defined for $|z|<1$ and can be analytically continued to the split complex $\mathbb{C}\setminus [1, \infty)$.
Jin et.al. \cite{jin2015analysis} proved the stability domain $\mathcal{S}_{\mathcal{L}1}$ for $\mathcal{L}$1 scheme is $A(\pi/4)$-stable by analyzing the function $F_\omega(z)$ directly.  See the definition below in \eqref{eq:numstab}.
Since $\mathcal{L}$1 scheme can be seen a fractional extension of the backward Euler scheme for classical ODEs and the backward Euler is $A$-stable, the above results in \cite{jin2015analysis} are not satisfactory and should be able to be improved. In \cite{jin2018discrete}, Jin et.al. further proved the $\mathcal{L}$1 scheme is $A((1-\alpha/2)\pi)$-stable, that is fractional $A$-stable, by making use of a very elaborate expansion formula for the polylogarithm function.

%

In the following, we study the stability domain of general $\mathcal{CM}$-preserving schemes and prove that they are at least $A(\pi/2)$ stable.
The results will allow us to applied $\mathcal{CM}$-preserving schemes to time fractional advection-diffusion equations, in which the eigenvalues of the space semi-discrete system lie in the left half complex plane but with nonzero imaginary part.
For the linear scalar test fractional ODE:
\begin{gather}\label{eq:linearode}
\mathcal{D}_c^{\alpha}u(t) =\lambda u(t)
\end{gather}
subject to $u(0)=u_{0}$ and $\lambda\in\mathbb{C}$, the true solution can be expressed as $u(t)=E_{\alpha}(\lambda t^{\alpha})u_0$, where $E_{\alpha}(z)=\sum_{k=0}^{\infty}\frac{z^k}{\Gamma(k\alpha +1)}$ is the Mittag-Leffler function. It is proved in \cite{lubich1986stability} that the solution satisfies that $u(t)\to 0$ as $t\to+\infty$
whenever 
\begin{gather}\label{eq:anastab}
\lambda\in \mathcal{S}^{*}:=\{z\in\mathbb{C}; z\neq0, |\arg(z)|>(\pi\alpha)/2 \}.
\end{gather}
Recall that the function $z\mapsto \arg(z)$ we use here has branch cut at the negative real axis and the range is in $(-\pi, \pi]$. 
Note that the stability region $\mathcal{S}^{*}$ for the true solution do not contain the point $z=0$. So does the numerical stability region $\mathcal{S}_h$ below.

Consider applying the $\mathcal{CM}$-preserving scheme with coefficients $a=(a_0, a_1,\cdots)$ to \eqref{eq:linearode} to obtain that
\begin{gather}\label{eq:linearstability}
u_n=u_0+\lambda h^{\alpha} [a*(u-u_0\delta_d)]_n,~~n\ge 0.
\end{gather}

\begin{definition}
The numerical stability region is defined by  
\begin{gather}\label{eq:numstab}
\mathcal{S}_h:=\{z=\lambda h^{\alpha}\in \mathbb{C}:  u_n\to 0~ \hbox{as}~ n\to +\infty \}.
\end{gather}
The numerical method is called $A(\beta)$-stable if the corresponding stability domain $\mathcal{S}_h$ contains the infinite wedge 
\begin{gather}\label{eq:Stheta}
S(\beta)=\{z\in\mathbb{C}; z\neq0, |\arg(-z)|<\beta \}.
\end{gather}
\end{definition}
We use $\arg(-z)$ here in order that the angle $\beta$ is counted from the negative real axis.
It is easy to find the generating function of the numerical solution sequence $\{u\}$ in \eqref{eq:linearstability} is given by
\begin{gather}\label{eq:genefun}
F_u(z)=u_0\frac{(1-z)^{-1}-\lambda h^{\alpha} F_a(z)}{1-\lambda h^{\alpha} F_a(z)}
=u_0 \left[1+\frac{z}{(1-\lambda h^{\alpha}F_a(z))(1-z)}\right].
\end{gather}
On the other hand, by Proposition \ref{pro:propertyomega}, for a $\mathcal{CM}$-preserving scheme 
\begin{gather}\label{eq:Fazasymp}
F_a(z)\sim (1-z)^{-\alpha},~z\to 1.
\end{gather}
Hence, if we can show 
\[
F_1(z):=\frac{z}{(1-\lambda h^{\alpha}F_a(z))(1-z)}
\]
is analytic in the region 
\begin{gather}
\Delta_{R,\theta}:=\{z\in\mathbb{C}: |z|\le R, z\neq 1, ~|\arg(z-1)|>\theta \}
\end{gather}
for some $R>1$ and $\theta\in (0, \frac{\pi}{2})$, then from Lemma \ref{lem:generating} we can find that
if $\lambda\neq 0$
\[
u_n\sim -\frac{u_0}{\lambda}h^{-\alpha}n^{-\alpha}\to 0,~~n\to +\infty.
\]
Hence, the domain 
\begin{gather}\label{eq:sdomain}
\mathcal{S}_1:=\left\{\zeta \in \mathbb{C}, \zeta\neq 0: \exists R>1, \theta\in \left(0, \frac{\pi}{2}\right), s.t.~  1-\zeta F_a(z) \neq 0,~\text{for } z
\in \Delta_{R,\theta} \right\}
\end{gather}
is contained in the stability region $\mathcal{S}_h$, i.e., $\mathcal{S}_1\subseteq \mathcal{S}_h$.

Let us start with region $\mathcal{S}_1$.
For the $\mathcal{CM}$ scheme, we have

\begin{lemma}\label{lem:complement}
Consider a scheme in \eqref{eq22} that is $\mathcal{CM}$-preserving. We have
\begin{gather}\label{eq:sdomainc}
\mathcal{S}_1^c=F_{\omega}\left(\overline{D(0, 1)} \right),
\end{gather}
where $\mathcal{S}_1$ is defined in \eqref{eq:sdomain}, $\omega=a^{-1}$ so that $F_{\omega}(z)=F_a^{-1}(z)$,  $\mathcal{S}_1^c$ is the complement of $\mathcal{S}_1$ and $D(0,1):=\{z\in \mathbb{C}:  |z|<1 \}$ is the open unit disk so that $\overline{D(0,1)}$ is the closed disk.
\end{lemma}

\begin{proof}
Since every $\Delta_{R,\theta}$ contains $\overline{D(0, 1)}\setminus\{1\}$
and $\mathcal{S}_1^c$ contains $0$, we must have 
\[
F_{\omega}\left(\overline{D(0, 1)} \setminus\{1\}\right) \subset \mathcal{S}_1^c.
\]
Since $F_{\omega}(1)=0$ by the asymptotic behavior of $F_a(z)$ in \eqref{eq:Fazasymp}, we thus conclude
\[
F_{\omega}\left(\overline{D(0, 1)}\right) \subset \mathcal{S}_1^c.
\]

On the other hand, for any $\zeta_0\notin F_{\omega}\left(\overline{D(0, 1)} \right)$ (thus $\zeta_0\neq 0$), we show that $\zeta_0\in \mathcal{S}_1$.
In fact, if not, for any $\Delta_{R_m,\theta_m}$, there exists $z_m\in \Delta_{R_m,\theta_m}$ such that
$F_{\omega}(z_m)=\zeta_0$. Consequently, we are able to find a sequence $\{z_m\}\subset F_{\omega}^{-1}(\zeta_0)$  with $z_i\neq z_j$ for $i\neq j$, and $|z_m|\to 1$. Hence, $\{z_m\}$ must have a limiting point $\bar{z}$. $\bar{z}\neq 1$ by \eqref{eq:Fazasymp}. Hence, $F_{\omega}(z)$ must be analytic around $\bar{z}$ so that $F_{\omega}(\bar{z})=\zeta_0$. This is a contradiction since $F_{\omega}(z)-\zeta_0$ is analytic, with zeros being isolated.
\end{proof}

From this lemma we can see that if we can prove some properties of the image of unit disk
under the map $F_{\omega}(z)=F_a^{-1}(z)=\omega_0-zG(z)$ for $z\in \overline{D(0,1)}$, where $G(z)$ is defined in \eqref{eq:Gfun}, we may get some information on the domain $\mathcal{S}_1$. With this observation, we have

\begin{theorem}\label{thm:stabilityregion}
Consider a $\mathcal{CM}$-preserving scheme for \eqref{eq1}. The complement of the numerical stability region 
$\mathcal{S}^c_h:=\mathbb{C}\setminus \mathcal{S}_h$ is a bounded set in the right half complex plane. 
There exists $\theta_0\in (0, \frac{\pi}{2})$ such that the numerical stability region $\mathcal{S}_h$ contains $S(\pi-\theta_0)$ defined in \eqref{eq:Stheta}, and also the wedge region 
\[
\bigcup_{\delta\le \delta_0} \left\{\zeta\in\mathbb{C}: |\zeta|\le \delta, |\arg(\zeta)|\ge \beta(\delta) \right\}
\]
for some small given positive constant $\delta_0>0$ and continuous function $\beta: [0, \delta_0]\to [0, \pi]$ such that $\beta(\delta)\to \frac{\alpha \pi}{2}$ as $\delta\to 0^{+}$.
In particular, the stability region contains the left half plane excluding $\{0\}$, i.e., $\mathcal{S}_h\supset \mathbb{C}^{-}\setminus \{0\}$, where $\mathbb{C}^{-}=\{\zeta\in\mathbb{C}:  \mathrm{Re} (\zeta)\leq0\}$.
\end{theorem}

The proof of Theorem \ref{thm:stabilityregion} relies on the following key observation of a completely monotone sequence
and its generating function:
\begin{lemma}
(\cite[Theorem 1]{lp16}) If a sequence $\{a\}$ is $\mathcal{CM}$, then 
there is a Hausdorff measure $\mu$ (nonnegative, supported on $[0, 1]$) such that
\[
a_n=\int_{[0, 1]} t^n d\mu(t),
\]
and consequently,
\begin{gather}\label{eq:Faintegralform}
F_a(z)=\int_{[0, 1]}\frac{1}{1-zt} d\mu(t),
\end{gather}
which is Pick, nonnegative on $(-\infty, 1)$.
\end{lemma}

 With the lemma, we now prove the main theorem of this part.

\begin{proof}[Proof of Theorem \ref{thm:stabilityregion}]

Since $a_0>0$, $\mu[0, 1]=a_0>0$.

We first show that $\mathcal{S}^c_h$ is bounded. 
Fix some $M>0$ large. Since $F_a(z)\sim (1-z)^{-\alpha}$ as $z\to 1$, for $\epsilon>0$ is small enough, in the domain $\overline{B(1, \epsilon)}\setminus[1,\infty)$, where $B(1,\epsilon):=\{\zeta\in\mathbb{C}: |\zeta-1|<\epsilon \}$, $|F_a(z)|>M$.
Note that on the region $D(0, 1)\setminus \overline{B(1,\epsilon)}$, $F_a(z)$ is an analytic function. Moreover, it is never zero since
it is a Pick function and positive on $(-\infty, 1)$ as $\mu[0, 1]>0$. Hence, $|F_a(z)|$ has a lower bound $C>0$. Hence, $\inf_{z\in \mathbb{C}\setminus(1,\infty)}|F_a(z)|>0$ and thus $\{\zeta: |\zeta|>C_1\}$ is contained in the stability region for some $C_1>0$ according to Lemma \ref{lem:complement}.

We now prove that $\mathcal{S}_h\supset S(\pi-\theta_0)$ (defined in \eqref{eq:Stheta}) for some $\theta_0\in (0, \frac{\pi}{2})$.
Consider $|z|\le R=1+\epsilon$.
If $\epsilon$ is very small, then $F_a(z)= (1+k(\epsilon))(1-z)^{-\alpha}$
for some function $k$ such that $k(\epsilon)\to 0$ as $\epsilon\to 0^+$. Hence, 
\begin{gather}\label{eq:argFa}
|\arg(F_a(z))| \le \alpha\frac{\pi}{2}+h(\epsilon),
\end{gather}
for some function $h$ satisfying that $h(\epsilon)\to 0$ as $\epsilon\to 0^+$.
When $z\in \overline{D(0,1)\setminus B(1,\epsilon)}$, then
$\mathrm{Re}(z)\le 1-\frac{1}{2}\epsilon^2<1.$
Using \eqref{eq:Faintegralform}, we know that $F_a(z)$ has positive real part,
so does $F_{\omega}(z)$.
Hence, we find that
\[
|\arg(F_\omega(z))| \le \frac{\pi}{2}-C(\epsilon),
\]
with $C(\epsilon)\to 0$  as $\epsilon\to 0^+$. 
Choosing suitable $\epsilon$, we further find
\begin{gather}\label{eq:argFw}
\sup_{z\in \overline{D(0,1)}\setminus \{1\}} |\arg(F_{\omega}(z))| \le  \theta_0 <\frac{\pi}{2}.
\end{gather}
Lemma \ref{lem:complement} then implies that the numerical stability region contains $S(\theta_0)$.

Regarding the last claim, we choose $\epsilon>0$ small and set $M_{\epsilon}=\sup_{z\in \overline{D(0,1)\setminus B(1, \epsilon)}} |F_{\omega}(z)|$. Then, for all $\zeta$ with $|\zeta|<1/M_{\epsilon}$, $F_{\omega}(z)=\zeta$ can only be possible for $z\in B(1, \epsilon)$. However, the phase of $F_{\omega}(z)=F_a^{-1}(z)$ in $B(1, \epsilon)$ is between $-(1+k(\epsilon))\frac{\pi \alpha}{2}$ and $(1+k(\epsilon))\frac{\pi \alpha}{2}$. 
This observation then leads to the claim regarding the asymptotic behavior of the stability region for $\zeta$ near the origin.
\end{proof}

As an immediate application of Theorem \ref{thm:stabilityregion}, we have the following.
\begin{corollary}\label{cor:stability}
Consider a $\mathcal{CM}$-preserving scheme for the test equation in \eqref{eq:linearode}.
If $|\arg(\lambda)|>\theta_0$, where $\theta_0$ is defined in Theorem \ref{thm:stabilityregion}, the scheme is unconditionally stable. If $|\arg(\lambda)|>\frac{\pi \alpha}{2}$, the scheme is stable for $h$ small enough.
\end{corollary}

Now a natural question is that whether the $\mathcal{CM}$-preserving schemes can be $A(\frac{\pi \alpha}{2})$ stable, that is, the numerical stability region contains the analytic stability region, $\mathcal{S}_h\supset S^{*}$, where $\mathcal{S}_h$ and $S^{*}$ are defined in \eqref{eq:anastab} and  \eqref{eq:numstab} respectively.
We point out that the above conjecture cannot be true in general. As a typical example, consider
\begin{gather}
F_a(z)=(1-z)^{-\alpha}+C\frac{1}{1-t_1 z},
\end{gather}
where $t_1$ is close to $1$. If the constant $C$ is large enough, the largest phase 
\[
\sup_{z\in \overline{D(0,1)}\setminus\{1\}}\arg(F_a(z))
\]
could be close to $\pi/2$. This function, however, also gives a consistent $\mathcal{CM}$-preserving scheme.

Hence, we can only hope some special scheme, like $\mathcal{L}$1 scheme, can achieve the better stability property.

\section{Monotonicity for scalar autonomous equations}
\label{sec:scalar}

It is noted that the solutions for classical first order autonomous one dimensional ODEs $u'=f(u)$ keeps the monotonicity, due to the facts of that the solution curves never cross the zeros of $f$ and hence $f(u)$ has a definite sign. In \cite{feng2018continuous}, the authors obtained a similar result for one dimensional autonomous fractional ODE
\begin{gather}\label{eq:autonomous}
\mathcal{D}_c^{\alpha}u=f(u),
\end{gather} 
where $t\mapsto u(t)\in \mathbb{R}$ is the unknown function.
\begin{lemma}[\cite{feng2018continuous}]\label{lem:monotonicity}
Consider the one dimensional autonomous fractional ODEs in \eqref{eq:autonomous}. Suppose that $f\in C^{1}(c, d)$ and $f'$ is locally Lipschitz on $(c, d)$. Then, the solution
$u$ with initial value $u(0)=u_0\in (c, d)$ is monotone on the interval of existence $(0, T_{max})$ ($T_{max}=\infty$ if the solution exists globally).  If $f(u_{0})\neq 0$, the monotonicity is strict. 
\end{lemma}

The basic idea in the proof of the above lemma is divided into two steps. First let $y(t)=u'(t)$ and write out the Volterra integral equations involving of $y$.
Then one can make use of the resolvent to transform the obtained integral equation into another new integral equation so that all the functions involved are non-negative. The positivity of the solution in the new integral equation leads to the required monotonicity. See the details in \cite{feng2018continuous}.

\subsection{General scalar autonomous equations}

In the following, motivated by Lemma \ref{lem:monotonicity}, we study the monotonicity of the solutions for one dimension (scalar) autonomous time fractional ODEs \eqref{eq:autonomous} obtained by the $\mathcal{CM}$-preserving numerical schemes.

\begin{theorem}\label{thm:generalmonotone}
Consider one dimension (scalar) autonomous time fractional 
ODEs \eqref{eq:autonomous}.   Suppose the numerical methods given in \eqref{eq22} or \eqref{eq24} is $\mathcal{CM}$-preserving. 
\begin{itemize}
\item If $f(\cdot)$ is $C^1$ and non-increasing, then for any step size $h>0$,
the numerical solution $\{u_n\}$ is monotone.

\item If $f(\cdot)$ is $C^1$ with $M:=\sup |f'(u)|<\infty$, then when $h^{\alpha}Ma_0<1$, $\{u_n\}$ is monotone.
\end{itemize}
\end{theorem}

From the following proof, we can see that for the second claim, we only need $M:=\sup |f'(u)|<\infty$ to be bounded on the convex hull of $\{u_n\}$ considered.
The proof is motivated by the time-continuous version in \cite{feng2018continuous}. 
We first prove a lemma about the discrete resolvent.

\begin{lemma}\label{lem:resolvent}
Suppose $a=\{a_n\}$ is completely monotone. For any $\lambda>0$, define the sequence $b=b(\lambda)$ given by
\[
b+\lambda (a*b)=\lambda a.
\]
Then, $b$ is completely monotone. In particular, it is nonnegative.
\end{lemma}

\begin{proof}
The generating function is 
\[
F_b(z)=\frac{\lambda F_a(z)}{1+\lambda F_a(z)}.
\]
Since $a$ is completely monotone, $F_a(x)\ge 0$ for $x<-1$, and thus so is $F_b(z)$. 

Moreover, we claim that $1+\lambda F_a(z)$ is never zero in the upper half plane. 
Since $a_0\ge 0$, then $1+\lambda F_a(z)\neq 0$ near $z=0$. If it is zero somewhere, 
then $F_a(z)$ is not a constant. By Lemma \ref{lem:Pick}, $\mathrm{Im}(F(z))>0$ for $z\in
\mathbb{C}_{+}$. This is a contradiction. Hence, $F_b(z)$ is analytic in the upper half plane.
Moreover, 
\[
F_b(z)=\frac{\lambda F_a(z)+\lambda^2|F_a(z)|^2}{|1+\lambda F_a(z)|^2}.
\]
Clearly, the imaginary part of $F_b(z)$ is nonnegative and hence it is Pick. The result follows from Theorem \ref{thm:CMequivalent}.
\end{proof}

\begin{proof}[Proof of Theorem \ref{thm:generalmonotone}]

For the convenience, we denote $f_j:=f(u_j).$ The scheme is written as
\begin{gather}\label{eq32}
u_n=u_0+h^{\alpha}\sum_{j=0}^{n-1}a_j f_{n-j}=h^{\alpha} [a*(f-f_0\delta_{d})]_{n},
\end{gather}
where $\delta_{d}=(1, 0, 0,\ldots)$ is the convolutional identity.
We define
$v_n:=u_{n+1}-u_n,  n\ge 0.$
Then, $v_n$ satisfies
\[
v_n=h^{\alpha}f_1 a_n+h^{\alpha}\sum_{j=0}^{n-1}a_j (f_{n+1-j}-f_{n-j}).
\]

We now define
$g_{n-j}:=\frac{f_{n+1-j}-f_{n-j}}{u_{n+1-j}-u_{n-j}} = \frac{f_{n+1-j}-f_{n-j}}{v_{n-j}} =f'(\xi_{n-j})$ for some $\xi$.
Then, the above equation is written as
\begin{gather}\label{eq33}
v_n=h^{\alpha}f_1 a_n+h^{\alpha}\sum_{j=0}^{n-1}a_j g_{n-j}v_{n-j}
=h^{\alpha}f_1 a_n+h^{\alpha}\sum_{j=0}^{n}a_j (g_{n-j}v_{n-j}-\delta_{n-j, 0}g_0v_0).
\end{gather}
In other words, we have that 
$v=h^{\alpha}f_1 a+h^{\alpha} a*(gv-\delta_d g_0v_0).$ Here we have made use of the notation 
$gv=\sum_{j=0}^{\infty}g_j v_j$.
Convolving this equation with $b$ defined in Lemma \ref{lem:resolvent}, we get that
\begin{gather}\label{eq34}
\begin{split}
b*v &=h^{\alpha}f_1 a*b+h^{\alpha}b*a*(gv-\delta_d g_0v_0)\\
&=h^{\alpha}f_1 a*b+h^{\alpha} \left(a-\frac{1}{\lambda}b\right)*(gv-\delta_d g_0v_0).
\end{split}
\end{gather} 
Consequently, it follows form (\ref{eq33}) and (\ref{eq34}) that 
$v_n-(b*v)_n=h^{\alpha}f_1 [a-a*b]_n+h^{\alpha}\frac{1}{\lambda} [b*(gv-\delta_d g_0v_0)]_{n}.$
Hence,
\[
v_n=h^{\alpha}f_1\frac{1}{\lambda}b_n+b_n v_0+ \left[ b*\left(v-v_0\delta+\frac{h^{\alpha}}{\lambda}(gv-\delta_d g_0 v_0)\right) \right]_{n}.
\]
Since $v_0=h^{\alpha}f_1a_0$, we further have
\[
v_n=h^{\alpha} \left( a_0+\frac{1}{\lambda} \right)f_1 b_n +\left[ b*\left( \left(1+\frac{h^{\alpha} g }{\lambda}\right) (v-v_0\delta_d)\right)  \right]_{n}.
\]
Hence, for $n\ge 1$,
\[
\left(1-b_0 \left(1+\frac{h^{\alpha}g_n}{\lambda} \right)\right)v_n
=h^{\alpha} \left( a_0+\frac{1}{\lambda}\right) f_1 b_n+\sum_{j=1}^{n-1}b_j \left(1+\frac{h^{\alpha} g_{n-j} }{\lambda} \right)v_{n-j}.
\]
Note that
$b_0=\frac{\lambda a_0}{1+\lambda a_0}<1.$ Now we discuss respectively in two cases.

Case 1: If $f$ is non-increasing, then we have that $1-b_0 \left(1+\frac{h^{\alpha}g_n}{\lambda} \right)>0$
for all $n$. Fix any $N>0$, we can always choose
$\lambda>0$ big enough such that $1+\frac{h^{\alpha}g_{n-j}}{\lambda}>0$ for all $j\le n\le N$. This choice will not
change the value of $u_j$ and thus $v_{n-j}$; it will only change $b_j$. On the other hand, we know from Lemma \ref{lem:resolvent} that $b_j$ for $j\geq 1$ are nonnegative. With this, we can see that the sign of $v_n=u_{n+1}-u_n$ keeps fixed and is the same as $f_1$
for all $n\le N$. Since $N$ is arbitrary, the claim is proved.

Case 2:  If $f$ has no monotonicity, but 
$M=\sup |f'|<\infty.$
We consider first that $
1-b_0(1+\frac{h^{\alpha}g_n}{\lambda})$. We  can require that
$1+\frac{h^{\alpha}g_n}{\lambda}<\frac{1}{b_0}=1+\frac{1}{\lambda a_0}$ such that $1-b_0 \left(1+\frac{h^{\alpha}g_n}{\lambda} \right)>0$.
Hence, we require
\begin{gather}\label{eq35}
h^{\alpha}Ma_0<1.
\end{gather}
If we choose $\lambda$ large enough,
$1+\frac{h^{\alpha}g_{n-j}}{\lambda}>0$
will also hold. Hence, the sign of $v_n$ is fixed.

\end{proof}

\begin{remark}\label{rem:remark}
If $u\in \bbR^d, d>1$ is a vector, applying the $\mathcal{CM}$-preserving numerical schemes to the equation (\ref{eq1}) 
does not necessarily imply $\|u_{n}\|$ to be monotone. See the example in numerical experiment.
However, if the system can be decomposed into $d$ orthogonal decoupled modes, in which the vector equation can essentially be equivalent to a set of scalar equations and then $\|u_{n}\|$ is monotone.
\end{remark}

\subsection{Linear equations with damping}
\label{sec:linearscalar}

If the equation in (\ref{eq1}) is one dimensional linear equation with damping, i.e., $f(u)=-\lambda u$ ($\lambda>0$), the result is much stronger. 
In fact, it is well known the solution can be expressed as 
\[
u(t)=u_{0}E_{\alpha}(-\lambda t^{\alpha}),
\]
where $E_{\alpha}(z)=\sum_{k=0}^{\infty}\frac{z^{k}}{\Gamma(k\alpha+1)}$ is the Mittag-Leffler function. We have that $u(t)$ is strictly monotonely decreasing and also $\mathcal{CM}$ due to the property of the Mittag-Leffler function $E_{\alpha}(z)$ \cite{gripenberg1990volterra}.
We can show that the corresponding numerical solution is also $\mathcal{CM}$. 

\begin{theorem}\label{thm:linear}
If the numerical method defined in (\ref{eq24}) is $\mathcal{CM}$-preserving, then for the scalar linear equations $\mathcal{D}_c^{\alpha}u=-\lambda u$ with $\lambda>0$ and $u_0>0$, the numerical solution $\{u_n\}$ is  a $\mathcal{CM}$ sequence. 
Moreover, the numerical solution goes to zero as $u_{n}\leq C (nh)^{-\alpha}$, 
where the constant $C$ is independent of $n$.
\end{theorem}

\begin{proof}
Taking the generating functions on the both sides of \eqref{eq23}, one has
\[
F_{\omega}(z)(F_u(z)-u_0(1-z)^{-1})=h^{\alpha}( F_f(z)-f_0)
=-\lambda h^{\alpha}(F_u(z)-u_0),
\]
where $F_f$ and $F_u$ denote the generating functions of $f=(f_0, f_1,\cdots)$ and $u=(u_0, u_1, \cdots)$ respectively. Then,
\begin{gather}\label{eq29}
F_u(z)=u_0\frac{F_{\omega}(z)(1-z)^{-1}+\lambda h^{\alpha}}{F_{\omega}(z)+\lambda h^{\alpha}}
=u_0 \left(1+\frac{(1-z)^{-1}-1}{1+\lambda h^{\alpha}F_a(z)}\right).
\end{gather}
The function 
\[
F_1(z):=1+\frac{(1-z)^{-1}-1}{1+\lambda h^{\alpha}F_a(z)}
\]
is clearly analytic on $(-\infty, 1)$ and nonnegative on $(-\infty, 1)$
(note that $F_a(x)\ge 0$ on this interval since $a$ is completely monotone).
Hence, we only need to check whether 
\begin{gather}\label{eq210}
G(z):=\frac{z}{ (1+\lambda h^{\alpha}F_a(z) )(1-z)}
\end{gather}
is a Pick function or not. Firstly, it is clearly analytic in the upper half plane by a similar argument in the proof of Lemma \ref{lem:resolvent}. 

Since $a$ is completely monotone, it is easy to see that
\[
(1, 0, 0,\cdots)+\lambda h^{\alpha}(a_0, a_1, \cdots)
=: (b_0, b_1, \cdots)
\] is also completely monotone. Consequently,  $(b_0-b_1, b_1-b_2,\cdots)$
is completely monotone.
Hence, if we define
\[
(1+\lambda h^{\alpha}F_a(z))(1-z)
=b_0-(b_0-b_1)z-(b_1-b_2)z^2-\cdots
=:b_0-z H(z),
\]
then $H(z)$ is a Pick function.
Consequently,
\[
G(z)=\frac{z}{b_0-zH(z)}=\frac{z(b_0-\bar{z}\bar{H}(z))}{|b_0-zH(z)|^2}.
\]
If $\mathrm{Im}(z)>0$, we find
\[
\mathrm{Im}(G(z))=\frac{1}{|b_0-zH(z)|^2} \left( b_0 \mathrm{Im}(z)-|z|^2 \mathrm{Im} \bar{H}(z) \right).
\]
Since $H$ is Pick, $\mathrm{Im} \bar{H}(z)=- \mathrm{Im} H(z)\le 0$. Hence,
$\mathrm{Im}(G(z))>0$.
This shows that $G$ is a Pick function. Therefore, $F_u(z)$ is also a Pick function for $u_0>0$.
This means that $u$ is completely monotone for $u_0>0$ and the claim follows.

Since $F_{a}(z)=(1+o(1))(1-z)^{-\alpha}$ as $z\to 1$, one has
\begin{gather}\label{eq211}
F_u(z)=u_0 \left(1+\frac{(1-z)^{-1}-1}{1+\lambda h^{\alpha}F_a(z)}\right),
\end{gather}
and thus
\[
F_u(z)\sim \frac{u_0}{\lambda h^{\alpha}} \frac{1}{(1-z)^{1-\alpha}} ~\hbox{as}~z\to 1. 
\]
Hence, taking $\beta=1-\alpha$ in Lemma \ref{lem:generating}, we get that $u_{n}\sim \frac{u_0}{\lambda h^{\alpha}} n^{-\alpha}$ as $n\to \infty$, which complete the proof.
\end{proof}

\begin{corollary}
Consider $\mathcal{D}_c^{\alpha}u=- Au$ for $u\in \mathcal{H}$, where $\mathcal{H}$ is a separable Hilbert space and $A: D(A)\to \mathcal{H}$ is nonnegative self-adjoint linear operator, with complete eigenvectors ($D(A)\subset\mathcal{H}$ is the domain of $A$). If we apply the $\mathcal{CM}$-preserving scheme to this equation, then the numerical solution $\|u_n\|$ is non-increasing.
\end{corollary}
In fact, let $\{e_k\}$ be the eigenvectors of $A$, then $\{e_k\}$ forms an orthogonal basis. One can possibly expand $u=\sum_{k=1}^{\infty}c_k(t)e_k$ such that the equation is decoupled into $\mathcal{D}_c^{\alpha}c_k(t)=-\lambda_k c_k(t)$, where $\lambda_k\ge 0$ is the $k$-th eigenvalue of $A$. Consequently, one has 
\begin{gather}\label{eq:norm}
\|u\|^2=\sum_{k=1}^{\infty}c_k^2(t)\|e_k\|^2,
\end{gather}
which is monotone by the conclusion from the scalar equation.
If we apply the $\mathcal{CM}$-preserving scheme to this equation, then the scheme is implicitly applied for each $c_k(\cdot)$ and \eqref{eq:norm} holds for the numerical solution as well. Then, Theorem \ref{thm:linear} gives the desired result. Typical examples include:
\[
\mathcal{D}_c^{\alpha}u=-(-\Delta)^{\beta}u,
\]
for $\beta\in (0, 1]$, and $\mathcal{H}=L^2(\mathbb{T}^d)$, where $(-\Delta)^{\beta}$ denotes the fractional Laplacian.

\section{Local truncation errors and convergence}
\label{sec:subdiffusion}

Let $u(\cdot)$ be the exact solution of the fractional ODE in \eqref{eq1} and $\mathcal{D}_h^{\alpha}u_n$ be the corresponding $\mathcal{CM}$-preserving numerical schemes in \eqref{eq22}.
In this section, we mainly focus on the local truncation error defined by
\begin{gather}\label{trerror}
r_n:=\mathcal{D}_h^{\alpha}u(t_n)-\mathcal{D}_c^{\alpha}u(t_n)=\mathcal{D}_h^{\alpha}u(t_n)-f(t_n, u(t_n))
\end{gather}
and the convergence of the scheme.

\subsection{Local truncation error}

 As well-known, if $f(t_0, u_0)\neq 0$, $u(\cdot)$ is not smooth
at $t=0$. In particular, $u(\cdot)$ is often of the form:
\begin{gather}\label{eq:solutionform}
u(t)=\sum_{m=1}^{M}\beta_m\frac{1}{\Gamma(m\alpha+1)} t^{m\alpha}+\psi(t),
\end{gather}
where $M=\lfloor 1/\alpha \rfloor$, $\beta_m$ are constants and $\psi(\cdot)\in C^1[0, T]$.
Hence, one cannot expect $\|r_n\|$ to be uniformly small. For example, 
if we apply the GL scheme to $u(t)=\frac{1}{\Gamma(1+\alpha)}t_+^{\alpha}$ corresponding to $f\equiv 1$, we have
\[
r_1=h^{-\alpha}\omega_0 \left(\frac{1}{\Gamma(1+\alpha)}h^{\alpha}-0\right)-1
=\frac{1}{\Gamma(1+\alpha)}-1,
\]
which does not vanish as $h\to 0^+$. However, we aim to show that when $n$ is large enough, $r_n$ is small, which allows us to establish the convergence for the typical solutions with weakly singularity at $t=0$ in \eqref{eq:solutionform} for fractional ODEs.

\begin{theorem}\label{thm:consistency}
Assume that $f(\cdot, \cdot)$ has certain regularity such that \eqref{eq:solutionform} holds for $t\in [0, T]$.  Let $h=T/N$ with $N\in\mathbb{N}$. We decompose
\[
r_n=r_n^{(1)}+r_n^{(2)},
\]
where $r_n^{(1)}$ is the truncation error corresponding to $m=1$
while $r_n^{(2)}=r_{n,m}^{(2)}+r^{(2)}_{n, \psi}$ corresponds to $m\ge 2$ and $\psi$. Then, $r_n^{(1)}$ is independent of $h$ but $\lim_{n\to\infty}r_n^{(1)}=0$, and
\[
\sup_{n: nh\le T}\|r_n^{(2)}\|=o(1),~h\to 0^+.
\]
\end{theorem}
\begin{proof}

We consider the truncation error on $\frac{1}{\Gamma(m\alpha+1)}t^{m\alpha}$, which is the fractional integral of $\frac{1}{\Gamma((m-1)\alpha+1)}t^{(m-1)\alpha}$. Clearly,
\[
\mathcal{D}_h^{\alpha} \left(\frac{1}{\Gamma(m\alpha+1)}t_n^{m\alpha} \right)
=h^{(m-1)\alpha}\sum_{j=0}^{n}\omega_j \frac{1}{\Gamma(m\alpha+1)}(n-j)^{m\alpha}=:h^{(m-1)\alpha}G_n,
\]
where $G_n$ is $n$th term of the convolution between $\omega$
and $\{\frac{1}{\Gamma(m\alpha+1)}n^{m\alpha}\}$, independent of $h$.
The generating function of $G$ is given by
\[
F_G(z)=F_{\omega}(z)\sum_{n=0}^{\infty}\frac{1}{\Gamma(m\alpha+1)}n^{m\alpha}z^n.
\]
By Proposition \ref{pro:propertyomega} and the asymptotic behavior of the generating function $\sum_{n=0}^{\infty}\frac{1}{\Gamma(m\alpha+1)}n^{m\alpha}z^n$ (see \cite[Theorem VI.7]{fs09} and the discussion below it), one has
\[
F_G(z)=(1+o(1))(1-z)^{\alpha} \left[(1+o(1))(1-z)^{-(m\alpha+1)}\right],~~z\to 1.
\]
 By  (2) of Lemma \ref{lem:generating}, we find when $m=1$,
$\lim_{n\to\infty}G_n=\lim_{z\to 1^-}(1-z)F_G(z)=1.$
We define $r_n^{(1)}$ to be the local truncation error corresponding to $m=1$:
\begin{gather}\label{eq:rn1}
r_n^{(1)}:=\beta_1 G_n-\beta_1\to 0,~n\to\infty.
\end{gather}

We now consider that $m\ge 2$. Using the first of Lemma \ref{lem:generating},
\[
G_n=(1+\varrho_n) \frac{1}{\Gamma((m-1)\alpha+1)}n^{(m-1)\alpha},
\]
where $\varrho_n$ are bounded and $\varrho_n\to 0$ as $n\to\infty$.
Hence, the truncation error corresponding to $m\ge 2$ is given by
\begin{gather}\label{eq:rn2a}
r_{n,m}^{(2)}:=\beta_m\frac{\varrho_n}{\Gamma((m-1)\alpha+1)}(nh)^{(m-1)\alpha}.
\end{gather}
If $N=T/h$ is big enough, this term is uniformly small. For $n\le \sqrt{N}$, it is controlled by $(\sqrt{N}h)^{(m-1)\alpha}$ while for large $n$, it is controlled by $T^{(m-1)\alpha}\sup_{n\ge \sqrt{N}}|\varrho_n|\to 0$ as $N\to\infty$.

Now, consider the local truncation error for $\psi$, which is $C^1[0, T]$.
To do this, we adopt some well-known consistent scheme for smooth functions, for example, the GL scheme \cite{lubich1986}
\[
\partial_h^{\alpha}\psi(t_n):=h^{-\alpha}\sum_{j=0}^n \bar{\omega}_j(\psi(t_{n-j})-\psi(0)).
\]
where $\bar{\omega}_j$ are the coefficients for GL scheme. Then, 
\begin{gather}\label{eq:rn2b}
r^{(2)}_{n, \psi}:=\left[ \mathcal{D}_h^{\alpha}\psi(t_n)-\partial_h^{\alpha}\psi(t_n) \right]
+[\partial_h^{\alpha}\psi (t_n)-\mathcal{D}_c^{\alpha}\psi(t_n)]=: R_{n,1}+R_{n,2}.
\end{gather}
By the well-known truncation error for GL for $\psi\in C^1[0, T]$, we have that $\sup_{n: nh\le T}\|R_{n,2}\|\le Ch^{\alpha}$, see for example \cite{jin2019numerical}. 
We now consider the first term $R_{n,1}$. It is in fact
\[
R_{n,1}=h^{-\alpha} \sum_{j=0}^n \gamma_j(\psi(t_{n-j})-\psi(0)),
\]
with $\gamma_j=\omega_j-\bar{\omega}_j=\varsigma_j (1+j)^{-1-\alpha}.$
By the asymptotic behavior in Proposition \ref{pro:propertyomega},
 $\varsigma_j$ is bounded and goes to zero as $j\to\infty$. 
 Fix $\epsilon>0$. We discuss in three cases.
 
 Case 1: $n\le h^{(\alpha-1)/2}$. We can control directly
 \[
 \|R_{n-1}\|\le h^{-\alpha} \sum_{j=0}^n |\gamma_j| \|\psi'(\xi_{n-j})\| t_{n-j} \le  Ch^{-\alpha} (nh) \sum_{j=0}^n |\gamma_j|
 \le Ch^{(1-\alpha)/2}.
 \]
 
 Case 2:  $h^{(\alpha-1)/2}<n\le \epsilon N$. Then, we can estimate directly that
 \[
 \|R_{n,1}\|\le 
 h^{-\alpha} \left\|\sum_{j=0}^n \gamma_j (\psi(t_{n-j})-\psi(t_n))\right\|
 +h^{-\alpha}\|\psi(t_{n})-\psi(0)\||\sum_{j=0}^n\gamma_j|
 \]
 The first term is controlled by $h^{-\alpha}\sum_{j=0}^n h(1+j)^{-\alpha}
 \le C (nh)^{1-\alpha}$.
 The second term is controlled due to $\sum_{j=0}^{\infty}\gamma_j=0$ by
 \[
 h^{-\alpha} (nh)  \left| \sum_{j=n+1}^{\infty}\gamma_j \right|
 \le C n h^{1-\alpha} n^{-\alpha} \le C (nh)^{1-\alpha}.
 \]
 Hence,  in this case $ \|R_{n-1}\|$ is controlled by $\epsilon^{1-\alpha}T^{1-\alpha}$.
 
 Case 3: $n\ge \epsilon N$. 
 We split the sum as
\begin{gather*}
\begin{split}
R_{n,1}=&h^{-\alpha} \sum_{j=0}^{\lfloor\epsilon N \rfloor} \gamma_j(\psi(t_{n-j})-\psi(t_n))
+h^{-\alpha} \sum_{j=0}^{\lfloor \epsilon N \rfloor} \gamma_j(\psi(t_{n})-\psi(0))\\
&+h^{-\alpha} \sum_{j=\lfloor\epsilon N \rfloor+1}^{n} \gamma_j(\psi(t_{n-j})-\psi(0)).
\end{split}
\end{gather*}
The first term is controlled directly by
$C h^{-\alpha}\sum_{j\le \lfloor\epsilon N \rfloor} jh (1+j)^{-1-\alpha}
\le C(\epsilon Nh)^{1-\alpha}.$
Note that $\sum_{j=0}^{\infty}\gamma_j=0$, the second and third 
can be estimated as
\begin{gather*}
\begin{split}
&h^{-\alpha}\left \|  -\sum_{j=\lfloor\epsilon N \rfloor+1}^{\infty} \gamma_j(\psi(t_{n})-\psi(0))
+ \sum_{j=\lfloor\epsilon N \rfloor+1}^{N} \gamma_j(\psi(t_{n-j})-\psi(0)) \right\|\\
&\le Ch^{-\alpha}\sum_{j=\lfloor\epsilon N \rfloor+1}^{\infty}|\varsigma_j|
(1+j)^{-1-\alpha}\le CT^{-\alpha}\epsilon^{-\alpha}\sup_{j\ge \lfloor\epsilon N \rfloor}|\varsigma_j|.
\end{split}
\end{gather*}
This goes to zero as $h\to 0^{+}$.
Hence, 
$\lim_{h\to 0}\sup_{n: nh\le T}\|R_{n,1}\|\le C(T)\epsilon^{1-\alpha}.$
Since $\epsilon$ is arbitrary, the limit must be zero.

Combining all the results, the claims are proved.
\end{proof}

\subsection{Convergence}

The $\mathcal{CM}$-preserving schemes have very good sign properties for the weight coefficients $\omega_j$, which allow us to prove stability and  also convergence. As pointed out in section \ref{sec:generalproperties}, if the scheme is $\mathcal{CM}$-preserving so that $\{a\}$ is completely monotone with $a_0>0$, then
\begin{gather}\label{sign}
\begin{split}
\hbox{(i):}~ \omega_{0}>0,~ \omega_{j}\leq 0~\hbox{for}~ j\geq1;\quad  \hbox{(ii):}~ \omega_{0}+\sum_{j=1}^{\infty}\omega_{j}\ge 0.
\end{split}
\end{gather}

We now conclude the convergence:
\begin{theorem}\label{thm:convergece}
Assume that $f(\cdot, \cdot)$ has certain regularity such that \eqref{eq:solutionform} holds for $t\in [0, T]$.  
If $f(t, \cdot)$ satisfies $(x-y)\cdot(f(t,x)-f(t,y))\le 0$ or is Lipschitz continuous,
then,
\begin{gather}
\lim_{h\to 0}\sup_{n: nh\le T}\|u(t_n)-u_n\|=0.
\end{gather}
\end{theorem}

\begin{proof}
Define
$e_n=u(t_n)-u_n.$
Then, we have
\[
\mathcal{D}_h^{\alpha}e_n=f(t_n, u(t_n))-f(t_n, u_n)+r_n,
\]
where $r_n$ is the local truncation error defined in \eqref{trerror}. Taking inner product on both sides with $e_n$ yields that 
\[
\mathcal{D}_h^{\alpha}\|e_n\|\le \|r_n\|+ \eta\|e_n\|,
\]
where $\eta=0$ if $f(t,\cdot)$ satisfies $(x-y)\cdot(f(t,x)-f(t,y))\le 0$ and $\eta=L$ be the Lipschitz constant if $f$ is Lipschitz.
Hence, we have
\[
\|e_n\|  \le \eta h^{\alpha}\sum_{j=0}^{n-1}a_j \|e_{n-j}\|
+h^{\alpha}\sum_{j=0}^{n-1}a_j \|r_{n-j}\|,~n\ge 1.
\]
We claim that
\begin{gather}
\epsilon_h:=\sup_{n: nh\le T} h^{\alpha}\sum_{j=0}^{n-1}a_j \|r_{n-j}\|=o(1),~h\to 0^+.
\end{gather}
We now do the same decomposition in Theorem \ref{thm:consistency} as $\|r_{n-j}\| \le \|r_{n-j}^{(1)}\|+\|r_{n-j}^{(2)}\|$. 
By this decomposition, the summation is controlled by
\[
h^{\alpha}\sum_{j=0}^{n-1}a_j\|r_{n-j}^{(1)}\|+h^{\alpha}\sum_{j=0}^{n-1}
a_j\|r_{n-j}^{(2)}\|.
\]

Let's separately estimate each term in the above equation.
For the second term, we have
 $$ h^{\alpha}\sum_{j=0}^{n-1}a_j\|r_{n-j}^{(2)}\| \leq C(nh)^\alpha\sup\limits_{j}\|r_j^{(2)}\|\leq CT^{\alpha}\sup\limits_{j}\|r_j^{(2)}\|=o(1), h\to 0^+,$$ 
 where we have used the property $h^{\alpha}\sum_{j=0}^{n-1}a_j\leq C(nh)^\alpha$, see Proposition \ref{pro:propertyomega}.
 The first term can be controlled by splitting technique as
\[
h^{\alpha}\sum_{j=0}^{n-N_1}a_j\|r_{n-j}^{(1)}\|
+h^{\alpha}\sum_{j=n-N_1}^{n}a_j\|r_{n-j}^{(1)}\|.
\]
For any $\epsilon>0$, we can pick $N_1$ fixed such that $\|r_{k}^{(1)}\|\le \epsilon$ for all $k\ge N_1$ when $N_1$ is big enough due to Theorem \ref{thm:consistency}.  The sum is then controlled by
\[
\epsilon h^{\alpha}\sum_{j=0}^{n-N_1}a_j+h^{\alpha}N_1^{1-\alpha} \le \epsilon t_{n-N_1}^{\alpha} +h^{\alpha}N_1^{1-\alpha} \le \epsilon T^{\alpha}+h^{\alpha}N_1^{1-\alpha}.
\]
Taking $h\to 0^+$, the limit is $T^{\alpha}\epsilon$. Since $\epsilon$
is arbitrarily small, the claim for $\epsilon_h$ is verified.

If $\eta=0$, the theorem is already proved. Now, we consider $\eta=L>0$.
To do this, we consider the auxiliary function $v(\cdot)$ which solves
$\mathcal{D}_c^{\alpha}v=L, ~~v(0)=2>0.$
Then, repeating what has been done, one can verify that
$v(t_n)=2+h^{\alpha}L\sum_{j=0}^{n-1}a_jv(t_{n-j})+\bar{\epsilon}_h.$
For $h$ small enough, $2+\bar{\epsilon}_h\ge 1$.
Hence, by the comparison principle (Proposition \ref{pro:comppf}), we find when $h$ is small enough,
\[
\|e_n\|\le \epsilon_h v(t_n)=2\epsilon_h E_{\alpha}(L t_n^{\alpha})\to 0,~h\to 0^{+},~~\forall nh\le T.
\]
The proof is completed.
\end{proof}

\subsection{Application to fractional diffusion equations}

 As a typical application to fractional PDEs, we consider the time fractional sub-diffusion equations, see \cite{liao2019discrete,lv2016error,stynes2017error,jin2019numerical,kopteva2019error}.
Here we follow the basic notation and idea from \cite{kopteva2019error} to establish the convergence of time semi-discretization problem using $\mathcal{CM}$-preserving schemes.

Let $\Omega\subset \mathbb{R}^d (d=1,2,3)$ be a bounded convex polygonal domain and $T>0$ be a fixed time. Consider the initial boundary value problem:
\begin{gather}\label{subdiffusion}
\begin{split}
&\mathcal{D}^{\alpha}_{c}u+\mathcal{L}u=f(x, t) \quad \hbox{for}\quad (x,t)\in \Omega\times (0, T],\\
&u(x,t)=0~ \hbox{for}~ (x,t)\in \partial\Omega\times (0, T],\quad u(x,0)=u_{0}(x)~ \hbox{for}~ x\in \Omega,
\end{split}
\end{gather}
where $\mathcal{D}^{\alpha}_{c}u$ denotes the $\alpha$ order of Caputo derivative with respect to $t$ and $\mathcal{L}$ is a standard linear second-order elliptic operator:
\begin{gather}\label{eq211}
\begin{split}
\mathcal{L}u=\sum_{k=1}^{d} \left\{-\partial_{x_k} (a_{k}(x) \partial_{x_k}u) +b_k(x)\partial_{x_k}u \right\}+ c(x)u,
\end{split}
\end{gather}
with smooth coefficients $\{a_k\}$, $\{b_k\}$ and $c$ in $C(\bar{\Omega})$, for which we assume that $a_k>0$ and $c-\frac{1}{2}\sum_{k=1}^{d}\partial_{x_k}b_k\geq 0$. We also assume that this equations there exists a unique solution in the given domain. 
Different from the classical integer order equations for $\alpha=1$, the solutions of fractional equations \eqref{subdiffusion} usually exhibit weak singularities at $t=0$, i.e., 
\begin{gather}\label{regularity}
\begin{split}
\| \mathcal{D}_{t}^{l}u \|_{L_{2}(\Omega)} \leq C(1+t^{\alpha-1})  \quad \hbox{for}\quad l=0,1,2,
\end{split}
\end{gather}
where $\mathcal{D}_t^l$ denote the classical $l$th order derivative with respect to time, see \cite{stynes2017error}. This low regularity of solutions at $t=0$ often leads to convergence order reduction for solution schemes. Many efforts have been made and new techniques developed to recover the full convergence order of numerical schemes, such as non-uniform grids \cite{liao2019discrete,stynes2017error,kopteva2019error}, and correction near the initial steps \cite{jin2019numerical}.

Consider the time semi-discretization of \eqref{subdiffusion} in time by $\mathcal{CM}$-preserving schemes 
\begin{gather}\label{semidiscretization}
\begin{split}
\mathcal{D}^{\alpha}_{h}U^n+\mathcal{L}U^n=f(\cdot, t_n)~ \hbox{in}~ \Omega,
~~~U^n=0~ \hbox{on}~\partial\Omega,\quad U^0=u_{0},
\end{split}
\end{gather}
where $U^n\approx u(x, t_n)$ and $\mathcal{D}^{\alpha}_{h}U^n=h^{-\alpha}\sum_{j=0}^n \omega_{j}(U_{n-j}-U_0)$ for  $n\geq 1$ stands for the $\mathcal{CM}$-preserving schemes with time step size $h>0$ as in \eqref{eq23}.

The good sign property in \eqref{sign} for $\mathcal{CM}$-preserving schemes will play a key role to establish the stability and convergence for scheme in \eqref{semidiscretization}. By using a complex transformation technique, the authors in \cite{lv2016error} obtain similar conditions like in \eqref{sign} and establish the stability and convergence for a $(3-\alpha)$-order scheme. We emphasize that the $\mathcal{CM}$-preserving schemes we present in this article naturally has this important property.

\begin{theorem}\label{subconvergence}
Let $u$ and $U^n$ be the solutions of equations \eqref{subdiffusion} and \eqref{semidiscretization} respectively. Then under the conditions $c-\frac{1}{2}\sum_{k=1}^{d}\partial_{x_k}b_k\geq 0$, we have that
\begin{gather}\label{convergence}
\begin{split}
\sup_{n: nh\le T}\|u(\cdot, t_n)-U^n\| \leq C h^{\alpha}\sup_{n: nh\le T}\sum_{j=1}^{n-1}a_j\|r_{n-j}\|\to 0, h\to 0^{+}.
\end{split}
\end{gather}
where $r_n= \mathcal{D}^{\alpha}_{h} u(\cdot, t_{n})-\mathcal{D}^{\alpha}_{c} u(\cdot, t_n)$ is the local truncation error. 
\end{theorem}

\begin{proof}
Let the error $e^n:=u(\cdot, t_n)-U^n$. It follows from \eqref{subdiffusion} and \eqref{semidiscretization} that $e^0=0$ and 
\[
\mathcal{D}^{\alpha}_{h}e^n+\mathcal{L}e^n
= \mathcal{D}^{\alpha}_{h} u(\cdot, t_{n})-f(t_n,\cdot)
=\mathcal{D}^{\alpha}_{h} u(\cdot, t_{n})-\mathcal{D}^{\alpha}_{c} u(\cdot, t_n), ~~1\leq n\leq T/h.
\]
By the definition $\mathcal{D}^{\alpha}_{h}e^n=h^{-\alpha}\sum_{j=0}^n \omega_{j}(e^{n-j}-e^0)$ the above equation can be rewritten as
\begin{gather}\label{errorequation}
\frac{\omega_0}{h^{\alpha}}  e^n+\mathcal{L}e^n=\frac{1}{h^{\alpha}} \sum_{j=1}^n (-\omega_{j})e^{n-j}+r_n, ~~1\leq n\leq T/h.
\end{gather}
Now we take the standard $L_2(\Omega)$ inner product in \eqref{errorequation} with $e^n$. Note that the condition $c-\frac{1}{2}\sum_{k=1}^{d}\partial_{x_k}b_k\geq 0$ implies that $\langle \mathcal{L}e^n, e^n\rangle_{L_2(\Omega)}\geq 0$. According to sign properties in \eqref{sign}, we get the error equation 
\begin{gather}\label{L2errorequation}
\frac{\omega_0}{h^{\alpha}}  \|e^n\|_{L^2(\Omega)} \leq \frac{1}{h^{\alpha}} \sum_{j=1}^n (-\omega_{j}) \|e^{n-j}\|_{L^2(\Omega)}+\|r_n\|_{L^2(\Omega)},~n\ge 1.
\end{gather}
In other words
\[
\mathcal{D}_h^{\alpha}\|e^n\|_{L^2(\Omega)}\le \|r_n\|_{L^2(\Omega)},~n\ge 1.
\]
The remaining proof is similar as Theorem \ref{thm:convergece}.
\end{proof}

From the above proof we can see that once we establish the order with respect to $\|r_n\|_{L^2(\Omega)}$, we will obtain the order of convergence of the numerical scheme. Similarly, for the fully discrete numerical schemes by applying a standard finite difference or finite element methods to spatial approximation to the time semi-discretization \eqref{semidiscretization}, we can also obtain the corresponding convergence order.

\section{Extension to Volterra integral equations}\label{Volterra}

We consider the second class of Volterra integral equation
\begin{gather}\label{eq41}
u(t)=u_{0}+\int_{0}^{t}k(t-s) f(s, u(s))ds,  \quad t>0,
\end{gather}
with initial value $u(0)=u_{0}$. 
We consider discretization
\begin{gather}\label{eq:voldis}
u_n-u_0=[b*(f-f_0\delta_{n,0})]_n=[b*f-f_0 b_n]_{n}=\sum_{j=0}^{n-1} b_j f_{n-j}, ~~n\ge 1.
\end{gather}
Note that here sequence $b$ corresponds to $h^{\alpha}a$ for the fractional ODE. We do not factor $h^{\alpha}$ out because $k(\cdot)$ may not be homogeneous. For example, $k(t)=t_+^{-1/2}+t_+^{-1/3}$.
We define the following.
\begin{definition}\label{eq:consistentvol}
We say the discretization given in \eqref{eq:voldis} is consistent for Volterra integral with $\mathcal{CM}$ kernel if a function $\phi(\cdot)$ with the typical regularity of $f(u(t))$ in \eqref{eq41} satisfies
\[
\epsilon_h:=\sup_{n\ge 1, nk\le T}\left\|\sum_{j=0}^{n-1}b_j\phi(t_{n-j})-\int_0^{t_{n}}k_{\alpha}(s)\phi(t_n-s)\,ds \right\|=o(1),~h\to 0^+.
\]
\end{definition}
\begin{definition}
We say a consistent (in the sense of Definition \ref{eq:consistentvol}) numerical method given in \eqref{eq:voldis}
for the convolutional Volterra integral equation \eqref{eq41} with  $\mathcal{CM}$ kernel is $\mathcal{CM}$-preserving if the sequence $b$ is a $\mathcal{CM}$ sequence.
\end{definition}

The main results regarding monotonicity given in Theorem \ref{thm:generalmonotone} for one dimension autonomous equations   can be extended to the Volterra integral equations with more general $\mathcal{CM}$ kernel functions directly. Moreover, the sign properties for the convolutional inverse $\nu:=b^{(-1)}$ also hold except that we generally have $\nu_0+\sum_{j=1}^{\infty}\nu_j\ge 0$ because $\|b\|_{\ell^1}$ may be finite. With the sign properties, analogy of Propositions \ref{pro:energy} and \ref{pro:comppf} hold except that we need $b_0 L<1$
to replace $h^{\alpha}La_0<1$.

\begin{theorem}\label{thm:convforvol}
Suppose \eqref{eq41} has a locally integrable $\mathcal{CM}$ kernel and $f(t, \cdot)$ is Lipschitz continuous. Then when applying a $\mathcal{CM}$-preserving scheme, we have
\[
\lim_{h\to 0^+}\sup_{n: nh\le T}\|u_n-u(t_n)\|=0.
\]
\end{theorem}
We sketch the proof here without listing the details. In fact, the error $e_n:=\|u(t_n)-u_n\|$ satisfies
\[
e_n\le L\sum_{j=0}^{n-1}b_j\|e_{n-j}\|+\epsilon_h,~n\le T/h.
\]
Consider $v(\cdot)$ solving $v(t)=2\delta+L\int_0^tk(t-s)v(s)\,ds$, with $\delta>0$. By the consistency,
\[
v(t_n)=2\delta+L\sum_{j=0}^{n-1}b_jv(t_{n-j})+\bar{\epsilon}(n,h)
\ge \delta+L\sum_{j=0}^{n-1}b_jv(t_{n-j}),
\]
when $h$ is small enough. Clearly, when $h$ is small enough,
$\epsilon_h<\delta$ for any fixed $\delta>0$. By direct induction,
\[
e_n\le v(t_n),~\forall n, nh\le T.
\]
The Volterra equation is continuous in terms of the initial value if the kernel is locally integrable. Since $\delta$ is an arbitrary positive number,  $\lim_{h\to 0}\sup_{n:nh\le T}e_n=0$.

\begin{remark}
When $k(t)=\frac{1}{\Gamma(\alpha)}t_+^{\alpha-1}$, the consistency in Definition \ref{def:consistency} can imply the consistency in Definition \ref{eq:consistentvol}. Hence, the conclusion in Theorem \ref{thm:convforvol} also applies to fractional ODEs. 
\end{remark}

Typical examples for completely monotone kernel functions are including that 

\begin{itemize}
 
\item The sum of several standard kernel: $k_{1}(t)=c_{1} k_{\alpha_{1}}(t)+c_{2}k_{\alpha_{2}}(t)+\cdots+c_{m}k_{\alpha_{m}}(t)$, where $c_j>0, \alpha_j\in (0, 1)$ for $j=1,2,...,m$.

\item The standard kernel with exponential weights: $k_{2}(t)= k_{\alpha}(t) e^{-\gamma t}, \gamma>0$.
\end{itemize}

One can easily construct $\mathcal{CM}$-preserving schemes for these equations using the ones in section \ref{sec:fourschemes}. In particular
\begin{enumerate}
\item for $k_1(t)$, one can use any scheme or their linear combination in section \ref{sec:fourschemes} to approximate $k_{\alpha_j}$ and this yields a $\mathcal{CM}$-preserving scheme for $k_1(t)$.
\item for $k_2(t)$, one can take the piecewise integral as approximation as in \cite{li2019discretization}:
\begin{gather}\label{eq42}
b_n=\int_{t_n}^{t_{n+1}}k_2(t)\,dt,
\end{gather}
where we recall $t_n=nh$.
\end{enumerate}

In addition, we can also use the CQ \cite{lubich1986} to calculate the convolutional Volterra integral. In general, we can approximate the convolutional integral as 
\begin{gather}\label{eq43}
\int_{0}^{t_n} k(t_n-s)g(s)ds\approx \left[ K \left(\frac{\delta(z)} {h} \right) F_{g}(z)\right]_{n},
\end{gather}
where $K$ is the Laplacian transform of the kernel $k(t)$, $\delta(z)=\breve{\rho}(z)/\breve{\sigma}(z)$ is the generating function based on classical linear multistep method $(\rho, \sigma)$ as in \eqref{eq26}, and $F_{g}(z)$ is the generating function of $(g_0, g_1,...)$. Therefore, if we can calculate $K$ accurately and choose $(\rho,\sigma)$ appropriately then we obtain the corresponding numerical schemes. As in section \ref{sec:fourschemes} for fractional ODEs, we can choose $(\rho,\sigma)$ in two ways:

(i): $\sigma(z)=z, \rho(z)=z-1$, and $\delta(z)=1-z$;

(ii): $\sigma(z)=\theta z+ (1-\theta), \rho(z)=z-1$ with $\theta\geq 1$, and $\delta(z)=\frac{1-z}{\theta + (1-\theta)z}=\frac{1-z}{2-z}$, where we take $\theta=2$.

For example, for $k_2(t)$ we have that 
\[
K[k_2(t)](z)=\mathcal{L} \left[k_\alpha(t) e^{-\gamma t} \right](z)= (z+\gamma)^{-\alpha}.
\]
Therefore, 
\begin{gather}\label{eq44}
\int_{0}^{t_n} k_{2}(t_n-s)g(s)ds\approx \left[ \left(\frac{\delta(z)} {h} +\gamma \right)^{-\alpha} F_{g}(z)\right]_{n}=h^\alpha \left[ \left(\delta(z) +h\gamma \right)^{-\alpha} F_{g}(z)\right]_{n}.
\end{gather}
Then we get the numerical schemes for Volterra integral equation \eqref{eq41} as
\begin{gather}\label{eq45}
u_{n}=u_0+h^{\alpha}\sum_{j=1}^{n}v_{n-j}f_{j},~~n\geq 1,
\end{gather}
where the weight coefficients $\{v_j\}$ derived from one of the following generating functions
\begin{gather}\label{eq46}
\begin{split}
&(i):  \left(1-z +h\gamma \right)^{-\alpha}= (1+h\gamma)^{-\alpha} \left(1-\frac{1}{1+h\gamma} z  \right)^{-\alpha}=\sum_{j=0}^{\infty}v_jz^{j};\\
&(ii):  \left(\frac{1-z} {2-z} +h\gamma \right)^{-\alpha}= \left(\frac{1+2h\gamma}{2}\right)^{-\alpha} \left(\frac{1-\frac{1+h\gamma}{1+2h\gamma} z} {1-z/2}  \right)^{-\alpha}=\sum_{j=0}^{\infty}v_jz^{j}.
\end{split}
\end{gather}

We now check if the generating functions $F_b(z)$ defined in \eqref{eq46} is a Pick function or not and the non-negativity on $(-\infty, 1)$.

For (i) in \eqref{eq46}, we have that $F_b(z)=\left(1-z +h\gamma \right)^{-\alpha}$. Since $\gamma>0$, it is easy to see $F_b(z)$ is a pick function and analytic, positive on $(-\infty, 1)$.

For (ii) in \eqref{eq46}, we have that $F_b(z)=\left(\frac{1-z} {2-z} +h\gamma \right)^{-\alpha}$. 
We rewrite
\[
F_b(z)=\left(\frac{1+2h\gamma}{2}\right)^{-\alpha} \left( \frac{1-z/2} {1-q z}  \right)^{\alpha}
:= \left(\frac{1+2h\gamma}{2}\right)^{-\alpha} (H(z))^{\alpha},
\]
where $q=\frac{1+h\gamma}{1+2h\gamma}\in(\frac{1}{2},1]$.
We now claim the function $H$ is Pick. 
 In fact,
\[
H(z)=\frac{1-z/2}{1-qz} =\frac{(1-z/2)(1- q\bar{z})}{|1- qz|^2}=\frac{1-q\bar{z}-z/2+q|z|^2/2} {|1-qz|^2},
\]
which implies that $\mathrm{Im}(H)=(q-\frac{1}{2})\mathrm{Im}( \frac{z}{|1- z|^2})$, and the result follows by noting that $q>\frac{1}{2}$.
Moreover, for $z\in \mathbb{R}$, the numerator becomes $1-(q+\frac{1}{2})z+\frac{q}{2}|z|^2$. 
Since  $1-(q+\frac{1}{2})z+\frac{q}{2}|z|^2=0$ has roots $z_{1}=2$ and $z_2=1/q>1$ so the numerator is positive on $(-\infty, 1)$ and the denominator is also positive on $(-\infty, 1)$,
so when $z\in (-\infty, 1)$, $H(z)>0$. Hence, $H(z)$ is a Pick function that is analytic and positive on $(-\infty, 1)$ and consequently, 
$F_b(z)$ is also Pick and nonnegative on $(-\infty, 1)$.

The weight coefficients $\{v_j\}$ can be recursively evaluated by the Miller formula in Lemma \ref{lem:FPS}.
Let that $\left(1-\frac{1}{1+h\gamma} z \right)^{-\alpha}=\sum_{j=0}^{\infty}m_j z^j$,  
$\left(1-\frac{1+h\gamma}{1+2h\gamma} z \right)^{-\alpha}=\sum_{j=0}^{\infty}n_j z^j$ and $\left(1- z/2 \right)^{\alpha}=\sum_{j=0}^{\infty}p_j z^j$,
where for coefficients $m_j$, $n_j$  and $p_j$ can be recursively computed by
\begin{gather}\label{eq47}
\begin{split}
m_0&=1, m_{k}=-\frac{1}{1+h\gamma}  \left(\frac{1-\alpha}{k} -1\right)m_{k-1}, \quad k\geq 1,\\
n_0&=1, n_{k}=-\frac{1+h\gamma}{1+2h\gamma} \left(\frac{1-\alpha}{k} -1\right)n_{k-1}, \quad k\geq 1,\\
p_0&=1, p_{k}=-  \frac{1}{2} \left(\frac{1+\alpha}{k} -1\right)p_{k-1}, \quad k\geq 1.
\end{split}
\end{gather}
Hence, the weight coefficients in schemes in \eqref{eq45} are given by 
\begin{gather}\label{eq48}
\begin{split}
(i): v_{j}=(1+h\gamma)^{-\alpha} m_{j}\quad \hbox{or}\quad (ii): v_{j}= \left(\frac{1+2h\gamma}{2}\right)^{-\alpha} \sum_{l=0}^{j} n_{j-l}p_{l}.
\end{split}
\end{gather}

Note that in the numerical scheme \eqref{eq45} for kernel $k_{2}(t)$, the coefficients $v_j$ depends on the step size $h$ explicitly. 
This is because the Laplacian transform of $k_2(t)$ is an inhomogeneous function on $z$ for $\gamma>0$, see \eqref{eq44}.

\section{Numerical experiments}\label{numerical}
In this section, we first perform numerical experiments to confirm the monotonicity of numerical solutions for $\mathcal{CM}$-preserving schemes applied to scalar autonomous fractional ODEs or Volterra integral equations with $\mathcal{CM}$ kernels. In \cite{wang2018long,wang2019dis}, the authors have shown that for linear scalar fractional ODEs with damping or delay differential equations, the long time decay rate $u_{n}=O(t_{n}^{-\alpha})$ as $n\to \infty$ both from theoretically and numerically by energy type methods. In this paper, we focus on the monotonicity of numerical solutions for nonlinear fractional ODEs and Volterra integral equations. We also provide numerical example on time fractional advection-diffusion equations to confirm the nice stability of $\mathcal{CM}$-preserving schemes.

\subsection{Fractional ODEs} Consider the scalar fractional ODE for $\alpha\in(0, 1]$,
\begin{gather}\label{eq:exam1}
\mathcal{D}_c^{\alpha}u(t) =Au-Bu^2,
\end{gather}
with initial value $u(0)=u_{0}$, where the two constants $A$ and $B$ satisfying that $A\cdot B>0$. For all order $\alpha\in(0, 1]$, this equation has two particular solutions $u_1=0$ and $u_2=\frac{A}{B}$. For $\alpha=1$ has the following general solution 
\[
u(t)=\frac{A}{B+\left(\frac{A}{u_0} -B\right) e^{-At}}.
\]
We can easily see from the expression that for $A, B>0$, if $u_0>0$, all the solutions asymptotically tend to the constant $A/B$; while for $u_0<0$, all the solutions will
blow up in finite time and have vertical asymptotic lines. The case for $A, B<0$ is similar. 

In Fig. (\ref{pic1}), we plot the numerical solutions for $\alpha=1$ and $\alpha=0.8$, respectively. It is clearly that all the solutions are monotone and asymptotically tends to the constant $A/B=2$, and they are asymptotic stable, as expected.  The order of $\alpha$ has a significant impact on the decay rates of the numerical solutions.  For the classical ODE with $\alpha=1$, we can see the solutions will decay exponentially while for $\alpha\in(0,1)$ the solutions will only decay with algebraic rate, which leads to the so called heavy tail effect for fractional dynamics \cite{wang2018long}.   

\begin{figure}[!ht]
\begin{center}
$\begin{array}{cc}
 \includegraphics[scale=0.50]{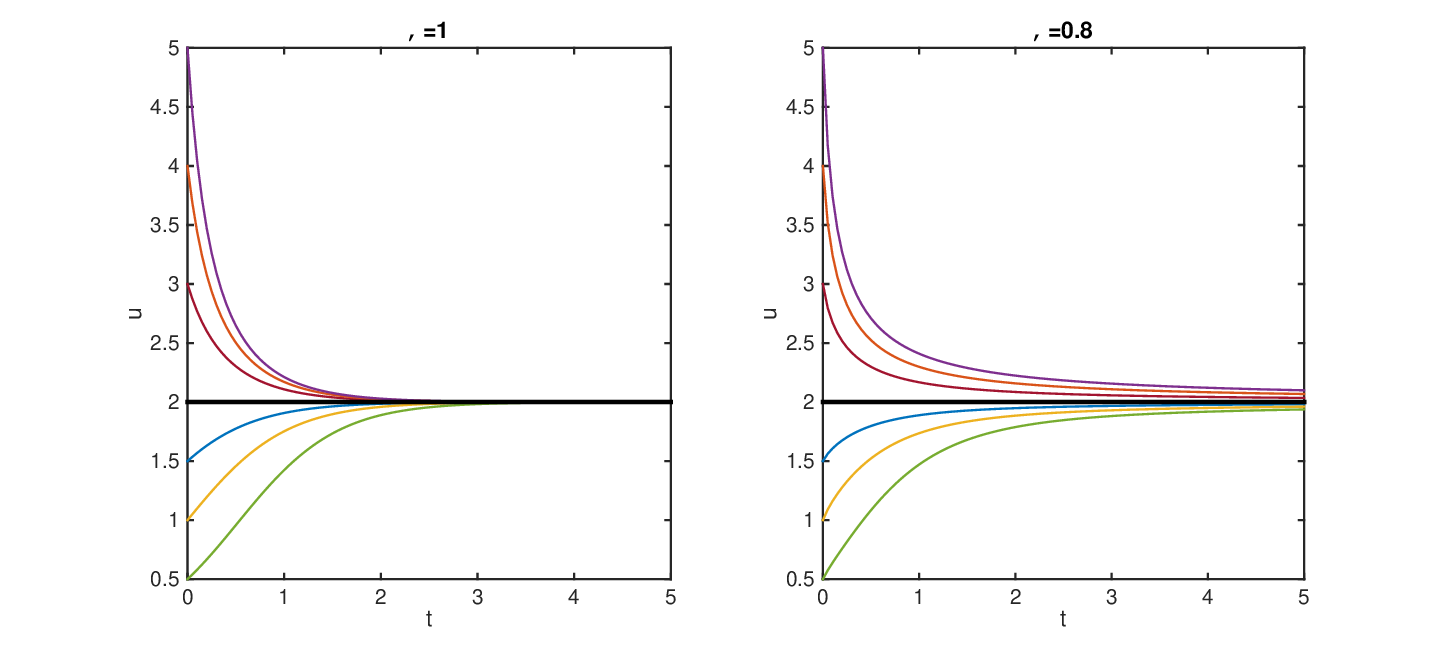}
 \end{array}$
\end{center}
\caption{Left: numerical solutions for $\alpha=1$ obtained by implicit Euler method; Right: numerical solutions for $\alpha=0.8$ obtained by Gr\"{u}nwald-Letnikov scheme. The initial values are taken as $0.5, 1, 1.5, 3, 4, 5$, respectively, and $h=0.05$, $T=5$ and $A=2, B=1$.}
\label{pic1}
\end{figure}

As pointed out in Remark \ref{rem:remark}, for general vector fractional ODEs in $\mathbb{R}^d$ with $d>1$, we can not expect the monotonicity of the Euclidean norm of the numerical solutions. Consider the fractional financial system \cite{petravs2011fractional}
\begin{equation*}
\begin{split}
\mathcal{D}_{c}^{\alpha}x(t)&=z(t)+(y(t)-1)x(t),\\
\mathcal{D}_{c}^{\alpha}y(t)&=1-0.1y(t)-x(t)^2,\\
\mathcal{D}_{c}^{\alpha}z(t)&=-x(t)-z(t).\\
\end{split}
\end{equation*}

\begin{figure}[!ht]
\begin{center}
$\begin{array}{cc}
 \includegraphics[scale=0.60]{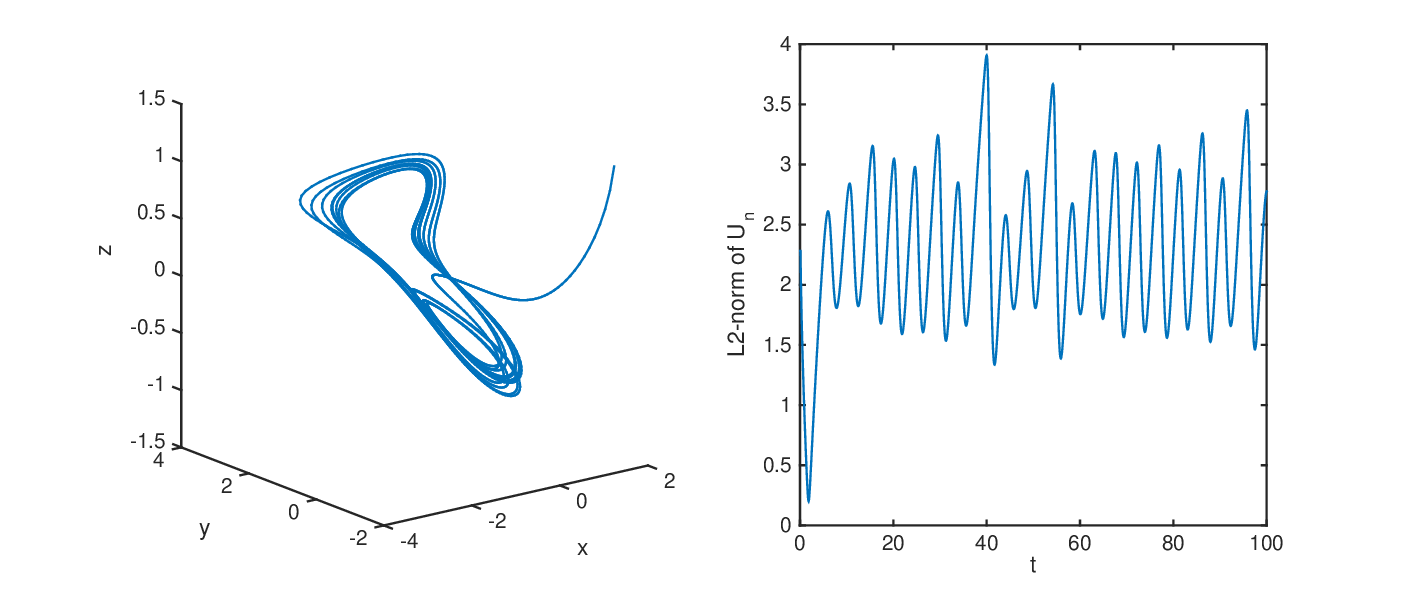}
 \end{array}$
\end{center}
\caption{Left: numerical solutions for $\alpha=0.9$ obtained by Gr\"{u}nwald-Letnikov scheme; Right: the $L2$-norm $\|U_{n}\|$, where $U=(x, y, z)^{T}$. The initial values $x_0=2, y_0=-1, z_0=1$, and $h=0.05$, $T=100$.}
\label{pic2}
\end{figure}

The fractional financial system is dissipative and there exists a bounded absorbing set \cite{wang2018long}. 
 Fig. (\ref{pic2}) shows that the solution doesn't tend to an equilibrium state, and of course $\|U_{n}\|$ doesn't have monotonicity, where $U=(x, y, z)^{T}$.
Numerical results obtained by other $\mathcal{CM}$-preserving schemes given in section \ref{sec:fourschemes} are very similarly,  and are not provided here.

\subsection{Volterra integral equations}
We study the monotonicity of numerical solutions for Volterra integral equation with $\mathcal{CM}$ kernel functions obtained by $\mathcal{CM}$-preserving schemes 
\begin{gather}\label{eq:exam2}
u(t)=u_{0}+\int_{0}^{t}k(t-s) f(u(s))ds,  \quad t>0,
\end{gather}
with initial value $u(0)=u_{0}$. Since the $\mathcal{CM}$ kernel $k_1(t)$ are very similar to the standard kernel $k_\alpha(t)$, we will focus on the kernel $k_{2}(t)=k_{\alpha}(t)e^{-\gamma t}$ for $\gamma>0$ in this example.  We consider the following three examples

(a) $f(u)=\lambda u$, $\lambda$ is a fixed parameter;

(b) $f(u)=Au-Bu^2$, where $A, B$ are parameters as in Example 1; 

(c) $f(u)=\sin(1+u^2)$.

In this example, we take the numerical schemes given in (ii) of \eqref{eq46} for the simulations for various initial values and parameters. 
The numerical results for scheme (i) of \eqref{eq46} are very similarly and not provided here.  
We take $h=0.1, T=10$ in all the following computations. 
The numerical solutions for (a) (b) and (c) are reported in Fig \eqref{pic3}, Fig \eqref{pic4} and Fig \eqref{pic5} respectively. 
The numerical results show that both the order $\alpha$ and parameter $\gamma$ will impact the decay rate and equilibrium state of the solutions significantly. 
But all numerical solutions for various initial values and parameters remain monotonic, as our theoretical results predicted.

\begin{figure}[!ht]
\begin{center}
$\begin{array}{cc}
 \includegraphics[scale=0.40]{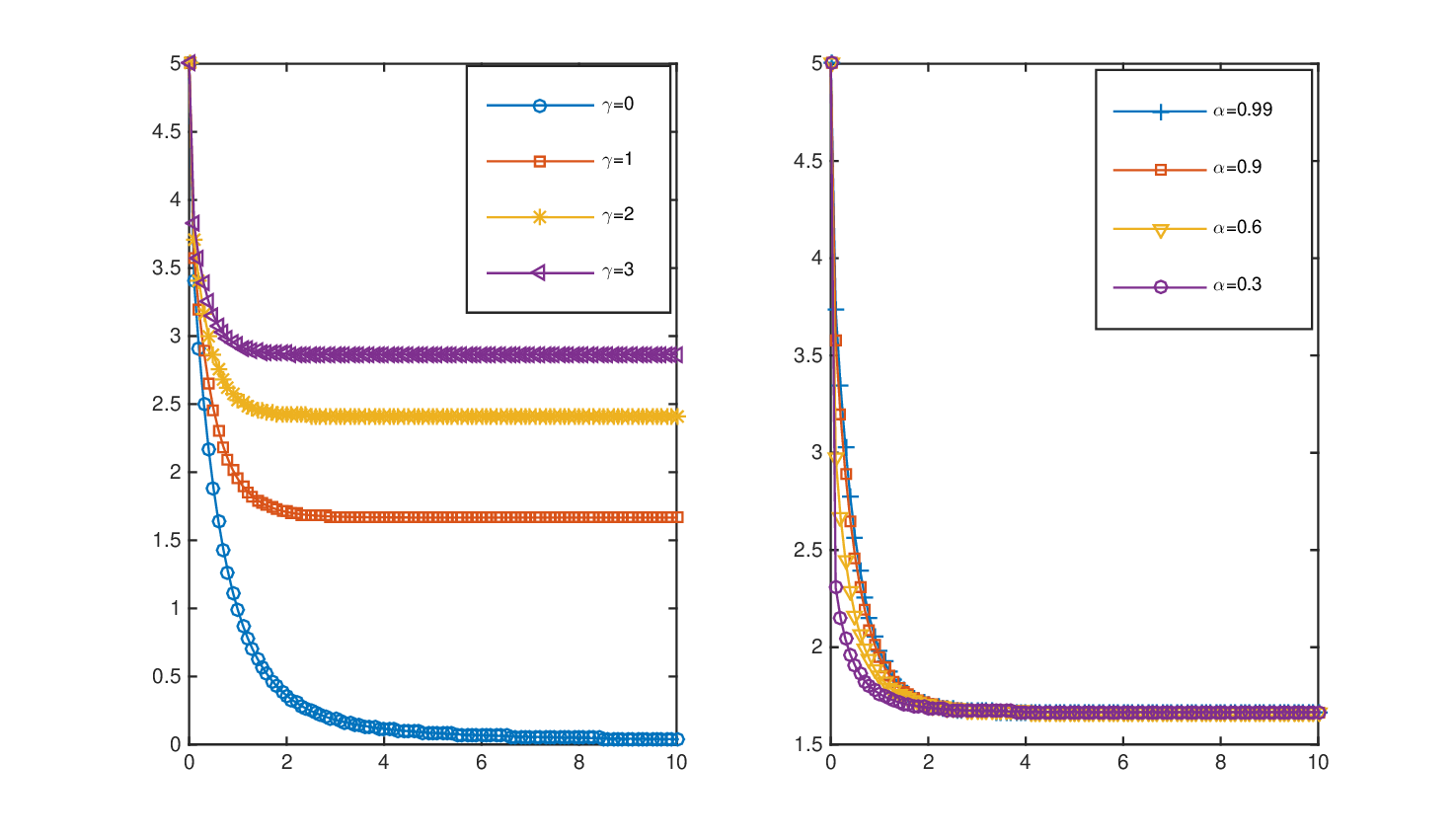}
 \end{array}$
\end{center}
\caption{Numerical solutions for (a) with $\lambda=-2$. Left: $\alpha=0.9$ and $\gamma=0, 1,2, 3$ respectively; Right: $\gamma=1$ and $\alpha=0.99, 0.9, 0.6, 0.3$ respectively.}
\label{pic3}
\end{figure}

\begin{figure}[!ht]
\begin{center}
$\begin{array}{cc}
 \includegraphics[scale=0.40]{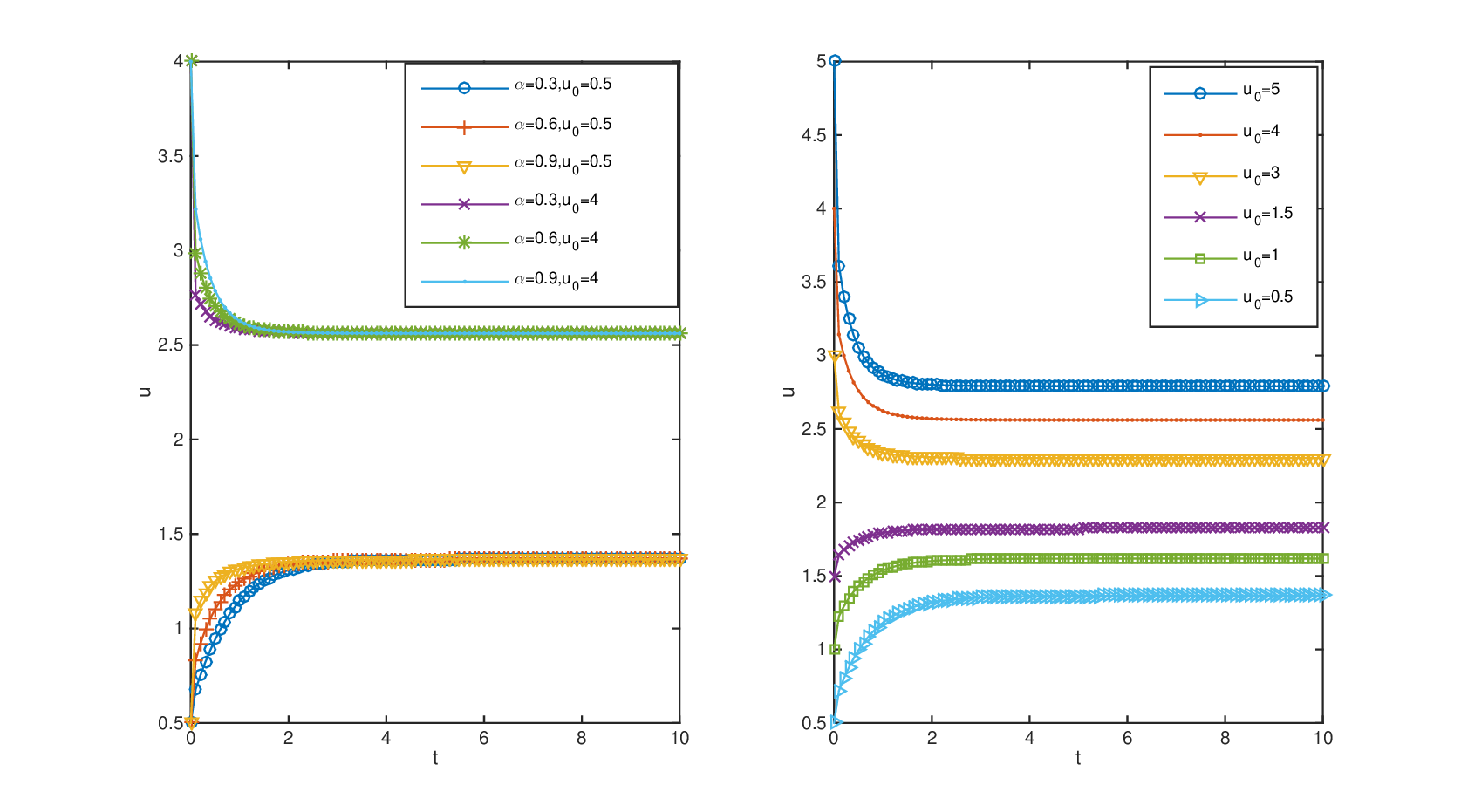}
 \end{array}$
\end{center}
\caption{Numerical solutions for (b) with $A=2, B=1$. Left: $\gamma=1$, $\alpha=0.9, 0.6, 0.3$ and $u_0=4, 0.5$ respectively; Right: $\gamma=1, \alpha=0.8$ and $u_0=0.5, 1, 1.5, 3,4,5$ respectively.}
\label{pic4}
\end{figure}

\begin{figure}[!ht]
\begin{center}
$\begin{array}{cc}
 \includegraphics[scale=0.40]{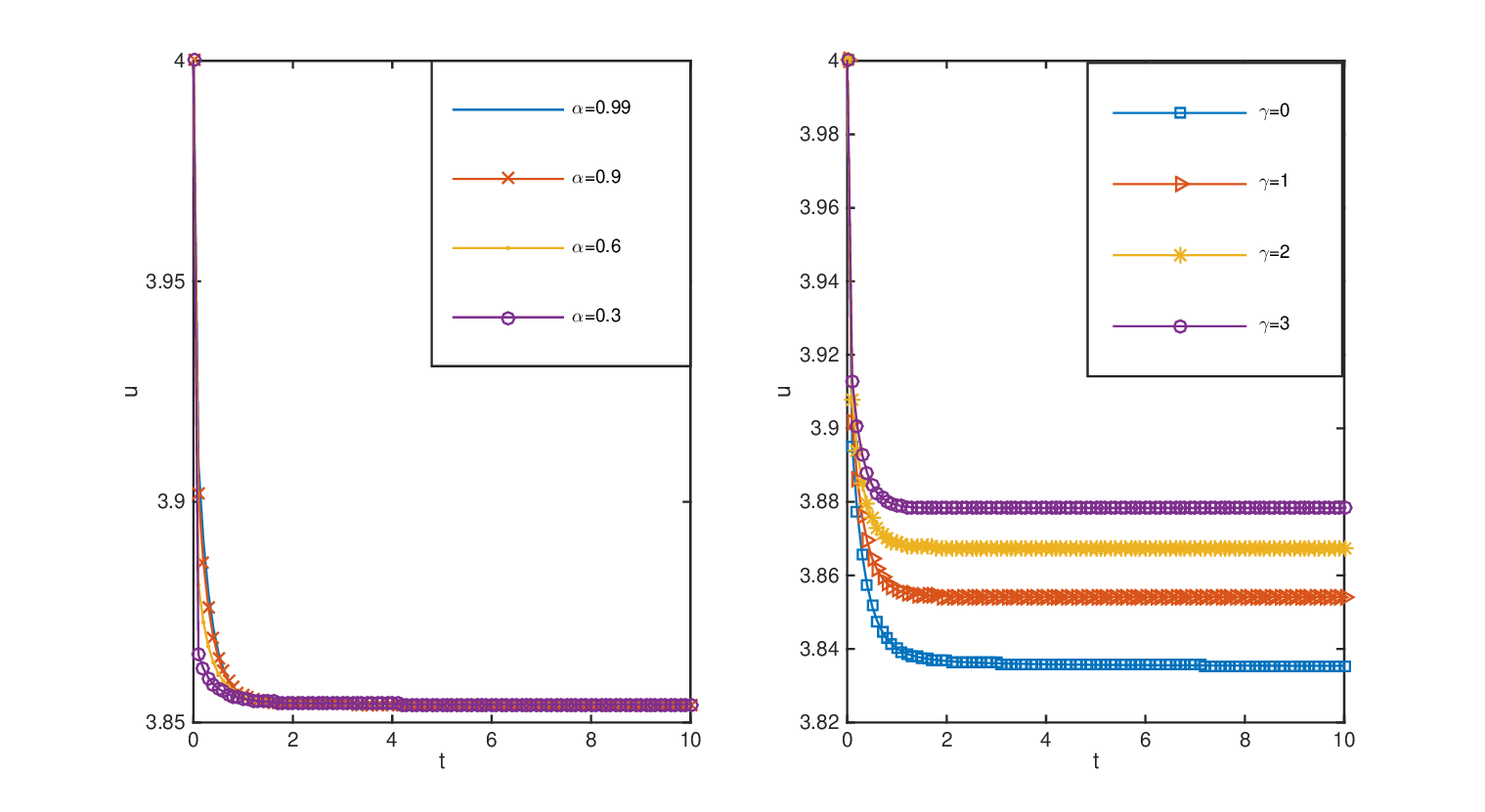}
 \end{array}$
\end{center}
\caption{Numerical solutions for (c). Left: $\alpha=0.99, 0.9,0.6, 0.3$ and $\gamma=1$ respectively; Right: $\gamma=0,1,2,3$ and $\alpha=0.9$ respectively.}
\label{pic5}
\end{figure}

\subsection{Application to fractional advection-diffusion equations}
\label{sec:advdiffu}
Consider the time fractional periodic advection diffusion problem 
\begin{equation}\label{eq:adiffu}
\begin{split}
_{~0}\mathcal{D}_{t}^{\alpha}u(x, t)+d u_{x}=Du_{xx}, ~t>0, x\in\Omega,\\
\end{split}
\end{equation}
with initial value $u(x, 0)=u_{0}(x)$ and Dirichlet or periodic boundary condition, where constant coefficients $d\in\mathbb{R}, D> 0$ and $\Omega\subset \mathbb{R}^{n} (n=1, 2, 3)$. 

When $d=0$, 
the equation \eqref{eq:adiffu} is reduced to the sub-diffusion equation, 
which has been thoroughly studied both mathematically and numerically in recent years. 
If $u_{0}(x)\in L^{2}(\Omega)$ and $u(x, t)=0$ for $x\in\partial\Omega$,
then it is proved in \cite{sakamoto2011initial} that the equation there exits a unique weak solution 
$u\in C([0, \infty]; L^2(\Omega))\cap C((0, \infty]; H^2(\Omega)\cap H^{1}_{0}(\Omega))$ and there exists a constant $C_{\alpha}>0$ such that 
\begin{equation} \label{eq:decay}
\begin{split}
\|u(\cdot , t)\|_{L^2(\Omega)}\leq \frac{C_{\alpha}}{1+\lambda t^{\alpha}} \|u_{0}\|_{L^2(\Omega)}, ~~\lambda>0, t>0.
\end{split}
\end{equation}

As we have pointed out earlier in Section \ref{sec:Introd}, the fractional sub-diffusion equations have two significant differences compared to the classical diffusion equations for $\alpha=1$. 
The first one is that the solution of model (\ref{eq:adiffu}) often exhibits weak singularity near $t=0$, i.e., 
$\|_{~0}\mathcal{D}_{t}^{\alpha} u(\cdot , t)\|_{L^2(\Omega)}\leq C_{\alpha} t^{-\alpha} \|u_{0}\|_{L^2(\Omega)}$ \cite{sakamoto2011initial}. 
In fact, this limited regularity makes it difficult to develop high-order robust numerical schemes and provide a rigorous convergence analysis on $[0, T]$ for some $T>0$.
Many efforts have been put on this problem and for the linear problems this problem has been well solved. 
Several effective high-order corrected robust numerical methods have been constructed and analyzed \cite{jin2019numerical,yan2018analysis,liao2019discrete,stynes2017error,kopteva2019error}.

The other one, which can be clearly seen from (\ref{eq:decay}), is the long time polynomial decay rate of the solutions,
 i.e., $\|u(\cdot , t)\|_{L^2(\Omega)}=O( t^{-\alpha})$ as $t\to +\infty$. 
 This is essentially different from the exponential decay of the solutions to a classical first order diffusion equations. 
 However, as far as we know, there is little work on studying the polynomial rate of the solutions and 
 characterizing their long tail effect for fractional sub-diffusion equations from the numerical point of view.
 In our recent work \cite{wang2018long}, we established the long time polynomial decay rate of the numerical solutions for a class of fractional ODEs
 by introducing new auxiliary tools and energy methods, which can also be used to characterize the numerical long time behavior of spatial semi-discrete PDEs as in (\ref{eq:adiffu}).

When $d=0$, the eigenvalues of fractional ODEs system obtained from space semi-discretization for fractional sub-diffusion equations are often negative real constants. Therefore, any time discrete numerical methods that contain the entire negative real half axis $(-\infty, 0]$ will lead to unconditionally stable schemes.

When $d\neq 0$, the corresponding eigenvalues of fractional ODEs system obtained from space semi-discretization have the form $\lambda_j=x_j+iy_j$, where $x_j, y_j$ are real constants and $x_j<0$. However, the constants $y_j$ are not zeros in general. 
In this case, if we still want to obtain an unconditionally stable numerical scheme in time direction, then the stable region of this scheme must contain the whole negative semi-complex plane $\mathbb{C}^{-}$. According to our results in this paper, the $\mathcal{CM}$-preserving schemes meet this stability requirement. 

As an example, we consider the one dimension fractional advection diffusion equation (\ref{eq:adiffu}) on $\Omega=[0, 1]$ with periodic boundary condition $u(0, t)=u(1, t)$.
For the space discretization on a uniform grid $\{x_1, x_2, ...,x_N\}$ with grid points $x_j = j\Delta x$ and mesh width $\Delta x = 1/N$, 
we use second-order central differences for the advection and diffusion terms. 
We obtain the semi-discrete system
\begin{equation}\label{eq:semis}
\begin{split}
_{~0}\mathcal{D}_{t}^{\alpha}u_{j}(t)+d \frac{u_{j+1}-u_{j-1}}{2\Delta x}=D \frac{u_{j+1}-2u_{j}+u_{j-1}}{{\Delta x}^2} , ~j=1,2,...,N,\\
\end{split}
\end{equation}
where $u_{0}=u_{N}, u_{N+1}=u_{1}$. For $\alpha=1$, this example has been carefully analyzed in \cite{verwer2004rkc, zbinden2011partitioned} and the corresponding eigenvalues can be obtained by standard Fourier analysis, which are given by  
\begin{equation}\label{eq:eigenv}
\begin{split}
\lambda_{j}^{\alpha}=\frac{2D}{{\Delta x}^2}(\cos(2\pi j\Delta x)-1)-i \frac{d}{{\Delta x}} \sin(2\pi j \Delta x), ~j=1,2,...,N.\\
\end{split}
\end{equation}
We can see those eigenvalues are located on the ellipse  in the left half plane $\mathbb{C}^-$: 
$\frac{ \left( x+\frac{2D}{\Delta x^2} \right)^2}{  \left(\frac{2D}{{\Delta x}^2} \right)^2}+ \frac{ y^2}{  \left(-\frac{d}{\Delta x} \right)^2}=1,$
which is centered at $\left(-\frac{2D}{\Delta x^2}, 0 \right)$ with two radii $\frac{2D} {{\Delta x}^2}$ and $\frac{d}{\Delta x}$, respectively. 
The stability results obtained in this paper show that any $\mathcal{CM}$-preserving schemes is $A(\pi/2)$-stable, so it can be used to solve the advection-diffusion fractional ODE \eqref{eq:semis}. 
\begin{figure}[!ht]
\begin{center}
$\begin{array}{cc}
 \includegraphics[scale=0.35]{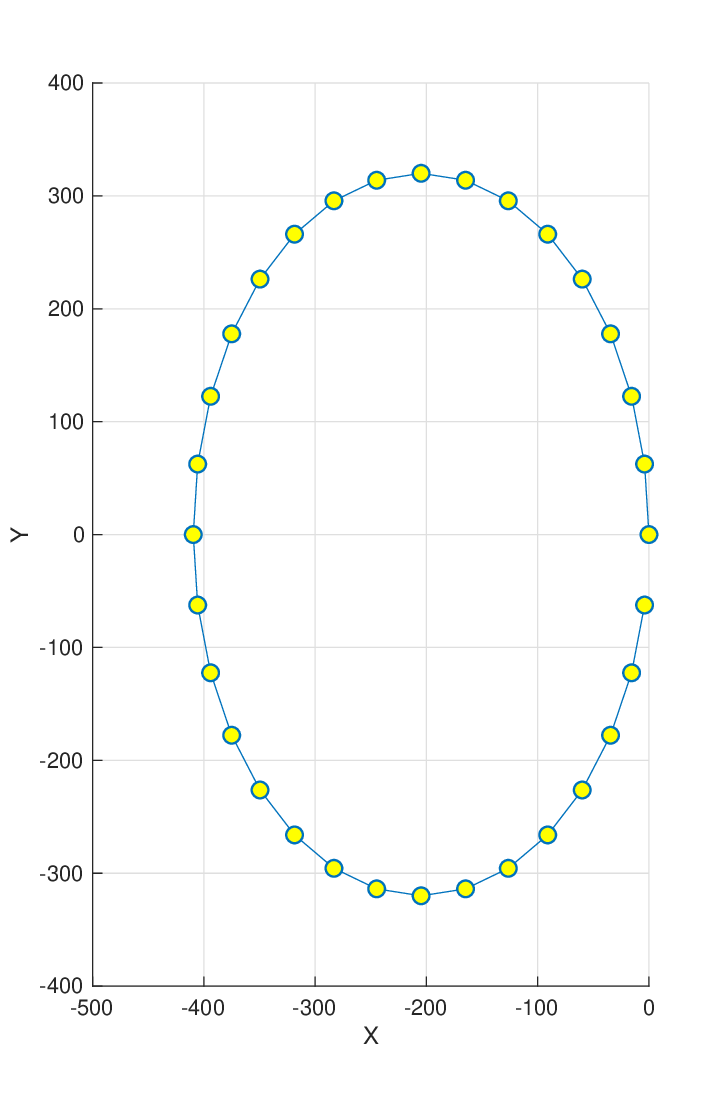}~~~~~
 \includegraphics[scale=0.65]{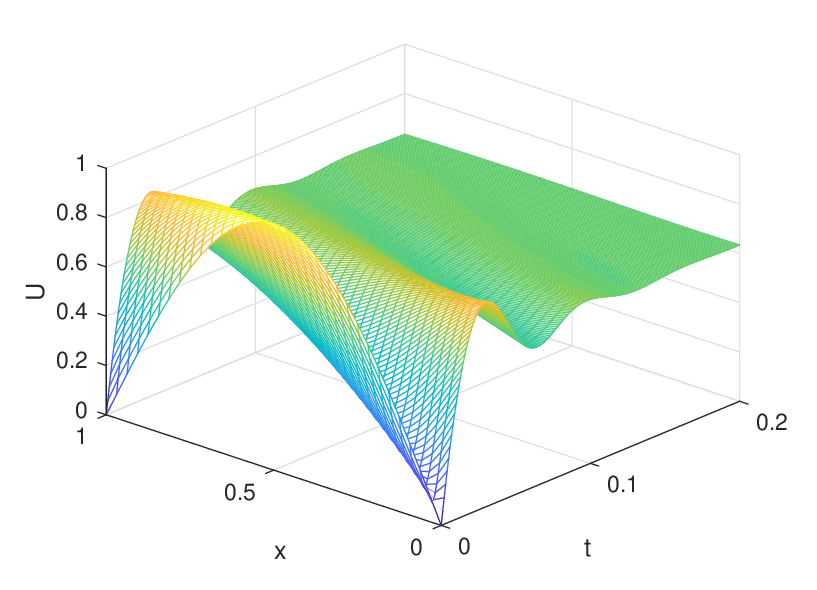}
 \end{array}$
\end{center}
\caption{The eigenvalues distributions in \eqref{eq:eigenv} and the numerical solutions for the semi-discrete system \eqref{eq:semis} with $d=10, D=0.1, \Delta x=1/32$.}
\label{Piceigen}
\end{figure}

As that in \cite{zbinden2011partitioned}, let the initial value  $U(0)\in\mathbb{R}^N$ for the semi-discrete fractional ODEs in \eqref{eq:semis} be 
\[
U(0)=\sum_{k=1}^{N}z_{k}\phi_{k}~ \hbox{with}~ z_k=\frac{1}{N} \sum_{k=1}^{N}u_0(x_j) \left( \overline{\phi_{k}} \right)_{j},
\]
where $\phi_{k}=\left( e^{2\pi i kx_{1}}, e^{2\pi i kx_{2}},...,e^{2\pi i kx_{N}} \right)^T \in\mathbb{C}^N$ stands for the discrete Fourier modes for $k=1,2,...,N$
and $U(t)=\left( u_{1}(t), u_{2}(t),...,u_{N}(t) \right)^T$ denotes the solution vector. Then the solution is given by 
\[
U(t)=\sum_{k=1}^{N}z_{k} E_{\alpha}(\lambda_{k} t^{\alpha})\phi_{k},
\]
where $E_{\alpha}(z)=\sum_{k=0}^{\infty} \frac{z^k}{\Gamma(k\alpha+1)}$ is the Mittag-Leffler function. 

In Figure \ref{Piceigen}, we plot the eigenvalues distributions and the corresponding numerical solutions obtained by $\mathcal{L}$1 scheme, which shows good numerical stability as long as the stable region is contained in the left half complex plane. Other $\mathcal{CM}$-preserving schemes give similar numerical performances and they are not provided here. 
Although the stable regions for some $\mathcal{CM}$-preserving schemes have been proved in other ways, we emphasize here that we can provide a unified framework to prove that they are all $A(\pi/2)$-stable and thus can be used for the time fractional advection-diffusion equations.

\section*{Acknowledgements} 
The work of L. Li was partially sponsored by NSFC 11901389, Shanghai Sailing Program 19YF1421300 and NSFC 11971314. 
The work of D. L. Wang was partially sponsored by NSFC 11871057, 11931013 and Project for Young Science and Technology Star of Shaanxi Province in China (2018 KJXX-070).

\appendix

\section{Proof of Proposition \ref{pro:comppf}}\label{app:compproof}

\begin{proof}[Proof of Propsosition \ref{pro:comppf}]

(1). Define the sequence $\xi=(\xi_n)$ by $\xi_n:=u_n-v_n$. Then, by the 
linearity of $\mathcal{D}_h^{\alpha}$,
\[
(\mathcal{D}_h^{\alpha}\xi)_n \le f(t_n, u_n)-f(t_n, v_n).
\]
where $( \cdot)_n$ stands for the $n$-th entry of the sequence.
Multiplying the indicator function $\chi_{(\xi_n\ge 0)}$ (i.e. the value is $1$ if $\xi_n\ge 0$ while the value is $0$ otherwise) on both sides of the inequality yields
\begin{gather*}
\begin{split}
&h^{-\alpha}\left(\omega_0 \xi_n \chi_{(\xi_n\ge 0)}
+\sum_{i=1}^{n-1}\omega_i \xi_{n-i}\chi_{(\xi_n\ge 0)}
-\left(\omega_0+\sum_{i=1}^{n-1}\omega_i \right)\xi_0 \chi_{(\xi_n\ge 0)} \right) \\
&\le [ f(t_n, u_n)-f(t_n, v_n)] \chi_{(\xi_n\ge 0)} \le 0.
\end{split}
\end{gather*}

We define $\eta_n=\xi_n \vee 0=\max(\xi_n, 0)$, i.e. the maximum between $\xi_n$ and $0$. Then, $\xi_n \chi_{(\xi_n\ge 0)}=\xi_n \vee 0=\eta_n$, 
$\xi_i \chi_{(\xi_n\ge 0)} \le \xi_i \vee 0 =\eta_i$ for any $i\neq n$.
Since $\omega_i\le 0$ and $-(\omega_0+\sum_{i=1}^n\omega_i)\le 0$, we then have
\begin{gather*}
\begin{split}
&\omega_0 \eta_n+\sum_{i=1}^n\omega_i \eta_{n-i}- \left(\omega_0+\sum_{i=1}^n\omega_i \right)\eta_0 \\
&\le \omega_0 \xi_n \chi_{(\xi_n\ge 0)}+\sum_{i=1}^n\omega_i \xi_{n-i}\chi_{(\xi_n\ge 0)}-\left(\omega_0+\sum_{i=1}^n\omega_i \right)\xi_0 \chi_{(\xi_n\ge 0)}.
\end{split}
\end{gather*}
Hence,
\[
(\mathcal{D}_h^{\alpha} \eta)_n \le 0.
\]
Clearly, $\eta_0= 0$, and by induction, it is easy to see $\eta_n\le 0$. This 
means $\eta_n=0$ and thus $\xi_n \le 0$. Similar argument applies
to $v_n$ and $w_n$, so we omit the details.

(2). The proof can be done by induction. 
We only compare $u$ with $v$. Comparing $v$ with $w$ is similar.

The condition gives $u_0\le v_0$. Suppose that for $n\ge 1$ we have shown $u_m \le v_m$ for all $m\le n-1$. We now prove $u_n\le v_n$.
Using again $\omega_0>0$, $\omega_i\le 0$ and $-(\omega_0+\sum_{i=1}^n\omega_i)\le 0$, we have
\[
h^{-\alpha}\omega_0(u_n-v_n)\le \mathcal{D}_h^{\alpha}(u-v))_n
\le f(t_n, u_n)-f(t_n, v_n)\le L|u_n-v_n|.
\]
Hence,
$u_n-v_n\le a_0 L h^{\alpha}|u_n-v_n|.$
If $a_0 L h^{\alpha}<1$, we must have $u_n-v_n\le 0$.

(3). The proof is similar as (2) by induction. One can in fact obtain $u_n-v_n\le a_0 L h^{\alpha}|u_n-v_n|$ using induction hypothesis. The argument is similar.

\end{proof}

\bibliographystyle{alpha}
\bibliography{fracgradient}

\begin{thebibliography}{ZTBK18}

\bibitem[BCM12]{bonaccorsi2012optimal}
S.~Bonaccorsi, F.~Confortola, and E.~Mastrogiacomo.
\newblock Optimal control for stochastic {V}olterra equations with completely
  monotone kernels.
\newblock {\em SIAM Journal on Control and Optimization}, 50(2):748--789, 2012.

\bibitem[Bru17]{brunner2017volterra}
H.~Brunner.
\newblock {\em Volterra integral equations: an introduction to theory and
  applications}, volume~30.
\newblock Cambridge University Press, 2017.

\bibitem[Cue07]{cuesta2007asymptotic}
E.~Cuesta.
\newblock Asymptotic behaviour of the solutions of fractional
  integro-differential equations and some time discretizations.
\newblock {\em Discrete and Continuous Dynamical Systems}, pages 277--285,
  2007.

\bibitem[DF02]{df02}
K.~Diethelm and N.J. Ford.
\newblock Analysis of fractional differential equations.
\newblock {\em J. Math. Anal. Appl.}, 265(2):229--248, 2002.

\bibitem[Die10]{diethelm10}
K.~Diethelm.
\newblock {\em The analysis of fractional differential equations: An
  application-oriented exposition using differential operators of {C}aputo
  type}.
\newblock Springer, 2010.

\bibitem[FLLX18]{feng2018continuous}
Y.~Y. Feng, L.~Li, J.-G. Liu, and X.~Q. Xu.
\newblock Continuous and discrete one dimensional autonomous fractional {ODE}s.
\newblock {\em Discrete \& Continuous Dynamical Systems-Series B}, 23(8), 2018.

\bibitem[FP19]{fornberg2019complex}
Bengt Fornberg and C{\'e}cile Piret.
\newblock {\em Complex Variables and Analytic Functions: An Illustrated
  Introduction}.
\newblock SIAM, 2019.

\bibitem[FS09]{fs09}
P.~Flajolet and R.~Sedgewick.
\newblock {\em Analytic combinatorics}.
\newblock Cambridge University press, 2009.

\bibitem[GG08]{galeone2008fractional}
L.~Galeone and R.~Garrappa.
\newblock Fractional adams--moulton methods.
\newblock {\em Mathematics and Computers in Simulation}, 79(4):1358--1367,
  2008.

\bibitem[GLS90]{gripenberg1990volterra}
G.~Gripenberg, S.-O. Londen, and O.~Staffans.
\newblock {\em Volterra integral and functional equations}, volume~34.
\newblock Cambridge University Press, 1990.

\bibitem[JLZ15]{jin2015analysis}
B.T. Jin, R.~Lazarov, and Z.~Zhou.
\newblock An analysis of the {L1} scheme for the subdiffusion equation with
  nonsmooth data.
\newblock {\em IMA Journal of Numerical Analysis}, 36(1):197--221, 2015.

\bibitem[JLZ17]{jin2017correction}
B.~Jin, B.~Li, and Z.~Zhou.
\newblock Correction of high-order {BDF} convolution quadrature for fractional
  evolution equations.
\newblock {\em SIAM Journal on Scientific Computing}, 39(6):A3129--A3152, 2017.

\bibitem[JLZ18]{jin2018discrete}
B.~Jin, B.~Li, and Z.~Zhou.
\newblock Discrete maximal regularity of time-stepping schemes for fractional
  evolution equations.
\newblock {\em Numerische mathematik}, 138(1):101--131, 2018.

\bibitem[JLZ19]{jin2019numerical}
B.~T. Jin, R.~Lazarov, and Z.~Zhou.
\newblock Numerical methods for time-fractional evolution equations with
  nonsmooth data: a concise overview.
\newblock {\em Computer Methods in Applied Mechanics and Engineering},
  346:332--358, 2019.

\bibitem[Kop19]{kopteva2019error}
N.~Kopteva.
\newblock Error analysis of the {L}1 method on graded and uniform meshes for a
  fractional-derivative problem in two and three dimensions.
\newblock {\em Mathematics of Computation}, 88(319):2135--2155, 2019.

\bibitem[LA14]{loy2014interconversion}
R.~J. Loy and R.~S. Anderssen.
\newblock Interconversion relationships for completely monotone functions.
\newblock {\em SIAM Journal on Mathematical Analysis}, 46(3):2008--2032, 2014.

\bibitem[LL18a]{liliu2018}
L.~Li and J.-G. Liu.
\newblock A generalized definition of caputo derivatives and its application to
  fractional {ODE}s.
\newblock {\em SIAM Journal on Mathematical Analysis}, 50(3):2867--2900, 2018.

\bibitem[LL18b]{li2018note}
L.~Li and J.-G. Liu.
\newblock A note on deconvolution with completely monotone sequences and
  discrete fractional calculus.
\newblock {\em Quart. Appl. Math}, 76(1):189--198, 2018.

\bibitem[LL18c]{liliu2018compact}
L.~Li and J.-G. Liu.
\newblock Some compactness criteria for weak solutions of time fractional
  {PDE}s.
\newblock {\em SIAM Journal on Mathematical Analysis}, 50(4):3963--3995, 2018.

\bibitem[LL19]{li2019discretization}
L.~Li and J.-G. Liu.
\newblock A discretization of {C}aputo derivatives with application to time
  fractional {SDE}s and gradient flows.
\newblock {\em SIAM J. Numer. Anal.}, 57(5), 2019.

\bibitem[LMZ19]{liao2019discrete}
H.~L. Liao, W.~McLean, and J.~W. Zhang.
\newblock A discrete gronwall inequality with applications to numerical schemes
  for subdiffusion problems.
\newblock {\em SIAM Journal on Numerical Analysis}, 57(1):218--237, 2019.

\bibitem[LP16]{lp16}
J.~Liu and R.~Pego.
\newblock On generating functions of {H}ausdorff moment sequences.
\newblock {\em Transactions of the American Mathematical Society},
  368(12):8499--8518, 2016.

\bibitem[Lub83]{lubich1983stability}
C.~Lubich.
\newblock On the stability of linear multistep methods for {V}olterra
  convolution equations.
\newblock {\em IMA Journal of Numerical Analysis}, 3(4):439--465, 1983.

\bibitem[Lub85]{lubich1985fractional}
C.~Lubich.
\newblock Fractional linear multistep methods for {A}bel-{V}olterra integral
  equations of the second kind.
\newblock {\em Mathematics of computation}, 45(172):463--469, 1985.

\bibitem[Lub86a]{lubich1986}
C.~Lubich.
\newblock Discretized fractional calculus.
\newblock {\em SIAM Journal on Mathematical Analysis}, 17(3):704--719, 1986.

\bibitem[Lub86b]{lubich1986stability}
C.~Lubich.
\newblock A stability analysis of convolution quadraturea for abel-volterra
  integral equations.
\newblock {\em IMA journal of numerical analysis}, 6(1):87--101, 1986.

\bibitem[Lub88]{lubich1988CQ}
C.~Lubich.
\newblock Convolution quadrature and discretized operational calculus. {I}.
\newblock {\em Numerische Mathematik}, 2(52):129--145, 1988.

\bibitem[LWZ19]{li2019linearized}
D.F. Li, C.D. Wu, and Z.M. Zhang.
\newblock Linearized galerkin {FEM}s for nonlinear time fractional parabolic
  problems with non-smooth solutions in time direction.
\newblock {\em Journal of Scientific Computing}, pages 1--17, 2019.

\bibitem[LX07]{lin2007finite}
Y.~Lin and C.J. Xu.
\newblock Finite difference/spectral approximations for the time-fractional
  diffusion equation.
\newblock {\em Journal of computational physics}, 225(2):1533--1552, 2007.

\bibitem[LX16]{lv2016error}
C.~W. Lv and C.~J. Xu.
\newblock Error analysis of a high order method for time-fractional diffusion
  equations.
\newblock {\em SIAM Journal on Scientific Computing}, 38(5):A2699--A2724, 2016.

\bibitem[PD97]{piero1997concepts}
G.~D. Piero and L.~Deseri.
\newblock On the concepts of state and free energy in linear viscoelasticity.
\newblock {\em Archive for Rational Mechanics and Analysis}, 138(1):1--35,
  1997.

\bibitem[Pet11]{petravs2011fractional}
I.~Petr{\'a}{\v{s}}.
\newblock {\em Fractional-order nonlinear systems: modeling, analysis and
  simulation}.
\newblock Springer Science \& Business Media, 2011.

\bibitem[SOG17]{stynes2017error}
M.~Stynes, E.~O'Riordan, and J.~Gracia.
\newblock Error analysis of a finite difference method on graded meshes for a
  time-fractional diffusion equation.
\newblock {\em SIAM Journal on Numerical Analysis}, 55(2):1057--1079, 2017.

\bibitem[SW06]{sun2006fully}
Z.Z. Sun and X.N. Wu.
\newblock A fully discrete difference scheme for a diffusion-wave system.
\newblock {\em Applied Numerical Mathematics}, 56(2):193--209, 2006.

\bibitem[SY11]{sakamoto2011initial}
K.~Sakamoto and M.~Yamamoto.
\newblock Initial value/boundary value problems for fractional diffusion-wave
  equations and applications to some inverse problems.
\newblock {\em Journal of Mathematical Analysis and Applications},
  382(1):426--447, 2011.

\bibitem[VSH04]{verwer2004rkc}
J.~G. Verwer, B.~P. Sommeijer, and W.~Hundsdorfer.
\newblock Rkc time-stepping for advection--diffusion--reaction problems.
\newblock {\em Journal of Computational Physics}, 201(1):61--79, 2004.

\bibitem[VZ15]{vergara2015optimal}
V.~Vergara and R.~Zacher.
\newblock Optimal decay estimates for time-fractional and other nonlocal
  subdiffusion equations via energy methods.
\newblock {\em SIAM Journal on Mathematical Analysis}, 47(1):210--239, 2015.

\bibitem[Wid41]{widder41}
D.V. Widder.
\newblock {\em The Laplace Transform}.
\newblock Princeton University Press, 1941.

\bibitem[WXZ20]{wang2018long}
D.~L. Wang, A.~G. Xiao, and J.~Zou.
\newblock Long-time behavior of numerical solutions to nonlinear fractional
  {ODE}s.
\newblock {\em ESAIM: Mathematical Modeling and Numerical Analysis},
  1(54):335--358, 2020.

\bibitem[WZ19]{wang2019dis}
D.~L. Wang and J.~Zou.
\newblock Dissipativity and contractivity analysis for fractional functional
  differential equations and their numerical approximations.
\newblock {\em SIAM Journal on Numerical Analysis}, 3(57):1445--1470, 2019.

\bibitem[Xu02]{da2002uniform}
D.~Xu.
\newblock Uniform $l$1 behaviour for time discretization of a {V}olterra
  equation with completely monotonic kernel: {I}. stability.
\newblock {\em IMA journal of numerical analysis}, 22(1):133--151, 2002.

\bibitem[Xu08]{da2008uniform}
D.~Xu.
\newblock Uniform $l$1 behavior for time discretization of a {V}olterra
  equation with completely monotonic kernel {II}: convergence.
\newblock {\em SIAM Journal on Numerical Analysis}, 46(1):231--259, 2008.

\bibitem[YKF18]{yan2018analysis}
Y.B. Yan, M.~Khan, and N.~J. Ford.
\newblock An analysis of the modified {L1} scheme for time-fractional partial
  differential equations with nonsmooth data.
\newblock {\em SIAM Journal on Numerical Analysis}, 56(1):210--227, 2018.

\bibitem[Zbi11]{zbinden2011partitioned}
C.~J. Zbinden.
\newblock Partitioned {R}unge--{K}utta--{C}hebyshev methods for
  diffusion-advection-reaction problems.
\newblock {\em SIAM Journal on Scientific Computing}, 33(4):1707--1725, 2011.

\bibitem[ZTBK18]{zeng2018new}
F.~H. Zeng, I.~Turner, K.~Burrage, and G.~E. Karniadakis.
\newblock A new class of semi-implicit methods with linear complexity for
  nonlinear fractional differential equations.
\newblock {\em SIAM Journal on Scientific Computing}, 40(5):A2986--A3011, 2018.

\end{thebibliography}

\end{document}